\documentclass[a4paper,11pt,fullpage]{article}
\usepackage[english]{babel}

\usepackage[left=1in,right=1in,top=1in,bottom=1in]{geometry}

\usepackage{amsmath}
\usepackage{amsthm}
\usepackage{amsfonts}
\usepackage{amssymb}
\usepackage{amsxtra}
\usepackage{latexsym}
\usepackage{ifthen}
\usepackage{cite}
\usepackage{esint}
\usepackage[cp1250]{inputenc}
\usepackage{xcolor}
\usepackage[hidelinks]{hyperref}

\usepackage{tikz}
\usepackage{pgfplots}
\pgfplotsset{compat=1.18}
\usetikzlibrary{patterns}
\usepgfplotslibrary{fillbetween}
\usepackage{graphicx,subfig,float,pgfplots}

\DeclareMathOperator*{\osc}{osc}

\numberwithin{equation}{section}
\newtheorem{theorem}{Theorem}[section]
\newtheorem{lemma}{Lemma}[section]
\newtheorem{remark}{Remark}[section]
\newtheorem{definition}{Definition}[section]

\newtheorem{propo}{Proposition}[section]

\def\XXint#1#2#3{{\setbox0=\hbox{$#1{#2#3}{\int}$}
     \vcenter{\hbox{$#2#3$}}\kern-.5\wd0}}

\begin{document}
\title{Regularity for Doubly Nonlinear Equations in the Mixed Regime
% H\"older Regularity for Mixed Singular-Degenerate Doubly Nonlinear Equations
}

\author{S. Ciani, E. Henriques, M. Savchenko, I. Skrypnik, Y. Yevgenieva}

\maketitle   

\begin{abstract} \noindent We study the local H\"older continuity of nonnegative solutions to doubly nonlinear equations by introducing a new technique that allows us to treat the cases where the equation is both singular and degenerate, up to specific Barenblatt numbers. Our argument relies on a new integral $L^1$-$L^1$ Harnack estimate, of independent interest. 
\end{abstract}

 % \begin{center}
%			\small
 %  \tableofcontents
%	\end{center}

\pagestyle{myheadings} \thispagestyle{plain}
\markboth{S. Ciani, E. Henriques, M. Savchenko, I. Skrypnik, Y. Yevgenieva}{Regularity for Doubly Nonlinear Equations in the Mixed Regime}

\section{Introduction}\label{Introduction}

%  \vskip0.3cm \noindent 
\subsection{Origins and framing of the topic}

The strong correspondence between the porous medium equation
and the $p$-Laplacian equation leads to the finding of a class of
equations enjoying the same properties: the so-called class of
 doubly nonlinear equations, which appear in the literature in three (somehow equivalent) forms
\begin{equation}\label{dnleq1}
u_t-\textrm{div} \big(|u|^{m-1} |Du|^{p-2} Du\big)=0 , \qquad p>1, \quad m>0 ,
\end{equation}

\begin{equation}\label{eq1.0}
u_t-\textrm{div} \big(|Du^m|^{p-2} Du^m \big)=0 , \qquad p>1, \quad m>0 ,
\end{equation}

\begin{equation}\label{dnleq3}
(u^q)_t-\textrm{div} \big(|Du|^{p-2} Du\big)=0 , \qquad p>1, \quad q>0  \ .
\end{equation}

\noindent Such equations rule several physical phenomena, e.g., the dynamics of
the turbulent flow of a non-Newtonian polytropic fluid through a
porous medium. These equations  were introduced, for the first time,
by Lions \cite{Li}. Especially in the last years,
many papers have been devoted to this topic (for the first contributions
on the subject see  \cite{Ka}). But, as already written above, the
reason for the interest in these equations consists in being
 a natural bridge between two natural generalisations of the heat
equation: the $p$-Laplacian and the porous medium equations. Each of the three formulations above presents a different characterisation of the loss of linear structure of the heat equation. 

\vskip0.2cm \noindent {\bf {\it Classification}}. A first classification of the doubly nonlinear equation (\ref{dnleq1}) was proposed in \cite{Vespri2} where the case $m+p>3$ is denoted as degenerate case, the case $2<m+p<3$ corresponds to the singular case, while $m+p=3$  is known as Trudinger's equation (it was introduced by Trudinger in \cite{T68}). Doubly nonlinear equations of type \eqref{eq1.0} are classified as doubly degenerate if $m > 1$ and $p > 2$, singular-degenerate if $m > 1$ and $1<p<2$, degenerate-singular if $0<m<1$ and $p > 2$, and doubly singular if $0<m<1$ and $1<p<2$. Furthermore, one can distinguish the diffusion process from slow diffusion, when $ m(p-1)>1$, to fast diffusion  when $m(p-1)<1$ (see \cite{BHSS21}).  In the latter formulation \eqref{dnleq3} (maintaining the nonlinearity in time), the loss of its parabolic character is classified according to what goes bad in time and space. Namely, the equation is said to be singular-degenerate when $0<q<1$ and $p>2$, degenerate-singular when $q>1$ and $1<p<2$, doubly singular when $0<q<1$ and $1<p<2$, and finally doubly degenerate when $q>1$ and $p>2$ (see  \cite{HenrLal13, EH20, EH21, EH23}). 
 All the different models \eqref{dnleq1}-\eqref{eq1.0}-\eqref{dnleq3} have in common the double nonlinearity and can informally be considered as a single one. The results can therefore be viewed {\it formally} as being imported from one to another with a simple change of variables. However, their different formulations lead to different definitions of weak solutions and particular technical requirements in order to obtain the expected, consistent, and correct estimates. 
 
 \vskip0.3cm 

\vskip0.2cm \noindent {\bf {\it Regularity Theory}}. Several authors have been developing a regularity theory for these classes of equations over the last decades, which is still fragmented. It is possible to identify two main approaches leading to the H\"older continuity: one relying on Cacciopoli estimates, an iterative scheme based on Moser's methods or another kind of iterative scheme based on De Giorgi's method and making use of intrinsic scaling (a powerful tool devised by DiBenedetto in the early 90s); and another different one based on a geometric tangential approach and intrinsic scaling technique (first presented for $u_t-\textrm{div} \big(|Du|^{p-2} Du\big)=f \in L^{r,q} ,$ $p>2$, see \cite{TU14}).

\noindent In what follows and to the best of our knowledge, we present several of the available results. 

\vspace{.1cm}

Concerning the equation \eqref{dnleq1}, in \cite{FornaroS} an intrinsic Harnack inequality is derived (and therefore the local H\"older continuity is obtained), within the range $m\geqslant 1$ and $p\geqslant 2$, for nonnegative local weak solutions of a wide class of doubly nonlinear degenerate parabolic equations having \eqref{dnleq1} as a prototype.  In \cite{FoSoVe, FoSoVe2, FoSoVe3} one can find energy estimates and an integral Harnack inequality for $m>1$ and $2<m+p<3$, results on the expansion of positivity, $L^r-L^\infty$ estimates, and Harnack type inequalities for  $2<m+p<3$. In \cite{FornGian10}, the authors present lower pointwise estimates, in terms of suitable sub-potentials and an alternative form of the Harnack inequality, for the nonnegative weak solutions of a wider class of doubly nonlinear degenerate ($p>2$, $m \geqslant 1$) parabolic equations having \eqref{dnleq1} as a prototype. The study of the behavior of weak solutions in the limiting case $p>1$ and $m+p = 2$ was done in \cite{FoHeVe, FoHeVe1, FoHeVe2} where the authors presented a weak Harnack estimate; proved $L^r$ and $L^r-L^\infty$ estimates and Harnack estimates within the very singular range $p>1$ and $3-p <m+p <2$; obtained stability results for the class of doubly nonlinear very singular parabolic equations. 

\noindent Results on the H\"older continuity can be found in \cite{Iv1} for $m-1=\sigma(p-1)$, $\sigma \geqslant 0$, $p\geqslant 2$; in \cite{Vesp92} for bounded weak solutions of a class of quasilinear parabolic equations having this equation as a prototype, for $m+p >2$ and $p>1$. In \cite{PoVe}, both interior and boundary H\"older continuity for bounded weak solutions to a class of quasilinear parabolic equations  were obtained within the ranges $m\geqslant 1$ and $p \geqslant 2$. In \cite{VeVe22}, the authors revisit several of the previous works and give alternative proofs, namely by  rewriting the equation as \eqref{eq1.0}, for $1<p<2$, $m>1$, and $2<m+p<3$. As for the nonhomegenous case $\displaystyle{u_t-\textrm{div} \big(|u|^{m-1} |Du|^{p-2} Du\big)=f\in L^{q,r} }$,	
in \cite{Araujo20} it is showed, for $m>1$ and $p>2$, that solutions are locally of class $C^{0,\alpha}$,  where H\"older's exponent $\alpha$ depends explicitly only on the optimal H\"older exponent for solutions of the homogeneous case, the integrability of $f$ in space and time, and the nonlinearity parameters $p$ and $m$. In \cite{ASRS23} the authors establish sharp regularity for the solutions in H\"older spaces, while in \cite{JSR22}, for $m\geqslant 1$ and $p \geqslant 2$, $C^{\alpha, \alpha/\theta}$ regularity estimates are derived for bounded weak solutions. In both papers, the proof is based on the geometric tangential method and the intrinsic scaling technique.

\vspace{.1cm}

When considering the doubly nonlinear prototype \eqref{eq1.0}, a Harnack inequality for nonnegative weak solutions was derived in \cite{BHSS21}, for $m> 0$, $p>1$, and $m(p-1)>1$. 

\vspace{.1cm}

Finally, when considering the $(q,p)$ doubly nonlinear setting presented by the prototype equation \eqref{dnleq3},
and as for Trudinger's equation, that is for $q=p-1$, a first proof, based on Moser's approach \cite{M64}, on Harnack inequality was given by Trudinger in \cite{T68}. In \cite{KiKu} the authors presented a simpler proof of Harnack inequality replacing the Lebesque measure by  a doubling Borel measure supporting a Poincar\'e inequality. In \cite{GiVe}, the authors considered $p>2$ and, by combining De Giorgi's methods with Moser's logarithmic estimates, showed that positive solutions satisfy a proper Harnack inequality. Recently, in \cite{BDGLS26} an intrinsic Harnack inequality was derived within the fast diffusion regime $0<p-1<q$.

\noindent The local H\"older continuity to Trudinger's equation was studied in \cite{KuusSiljUrba10}, for $p>2$, and in \cite{KLSU}, for $1<p<2$, by discriminating between large scales (for which a Harnack inequality is used) and small scales (using intrinsic scaling). When considering possibly sign-changing solutions, for $p>1$, the authors \cite{BDL21} established the interior and boundary H\"older continuity as well as an alternative proof of Harnack inequality for nonnegative solutions which leads to a Liouville-type result. The nonhomegeneous case $\displaystyle{(u^{p-1})_t-\textrm{div} \big(|Du|^{p-2} Du\big)=f \in L^{q,r}}$, $p>2$, was studied in \cite{NU20}, where the authors showed that bounded solutions are locally H\"older continuous using the full power of the homogeneity in the equation to develop the regularity analysis in the $p$-parabolic geometry, without any need of intrinsic scaling, as anticipated by Trudinger. Several regularity results were also obtained when $q\neq p-1$. In \cite{HenrLal13}, the authors prove local H\"older continuity for the degenerate-singular range $q>1$ and $1<p<2$. In \cite{EH20} the same regularity result is obtained now for $p>2$ and $0<q<1$. The expansion of positivity was proved in \cite{EH21} (and afterwards also in \cite{M23}), and by exploiting these results the H\"older continuity has been derived in \cite{EH23} for $1<p<2$ and $p-1<q<1$, and in \cite{CH25} for $p>2$ and $0<q<p-1$. Within the range $1<p<2$ and $q>p-1$, local H\"older continuity of possibly sign-changing solutions was proved in \cite{NaS22} using  the method of intrinsic scaling and expansion of positivity (without the exponential shift). The case $p>2$ and $0<q<p-1$ was treated in \cite{BDLS23}, where the boundary regularity for initial-boundary value problems of Dirichlet and Neumann type was also  obtained. 

\noindent Local and global boundedness results are presented in \cite{HL16, HL18} by working in measure spaces equipped with a doubling non-trivial Borel measure supporting a Poincar\'e inequality for $0<q<1$, $p>1$ and $q>1$, $p>1$, respectively.

\vskip0.2cm \noindent {\bf {\it The issue of H\"older Continuity}}. If $u$ is a solution to equation \eqref{eq1.0} and $C$, $c$ are constants such that $(C-u)$ and $(u-c)$ are nonnegative, unfortunately the functions $(C-u)$ and $(u-c)$ do not solve an equation that has a similar structure to the original \eqref{eq1.0}. This simple fact destroys the usual implication {\it Harnack inequality} $\Rightarrow$ {\it H\"older continuity}, obtained by controlling the oscillation of $u$ by properly choosing $C,c$ as the local supremum or infimum of $u$ in an iterative fashion. The classic method therefore has to distinguish between two cases (see \cite{CV} for a simple example):
\begin{enumerate}
    \item when the infimum of $u$ is quantitatively away from zero, which means that locally on a cylinder $Q$ the inequality
\[\inf_Q u \geqslant\frac{1}{\gamma} \sup_Q u\] is valid\footnote{We observe that this assumption already embodies a Harnack inequality for a nontrivial $u$ in $Q$, that is indeed useless without further properties of the equation.} for some fixed constant $\gamma>1$;
\item when the supremum of $u$ is quantitatively controlled by the oscillation, which means that locally on a cylinder $Q$ the inequality
\[\sup_Q u \leqslant  \tilde{\gamma} \osc_Q u\] holds true for some fixed constant $\tilde{\gamma}(\gamma)>1$.
\end{enumerate}
One can easily confirm that these two alternatives constitute a dichotomy (see for instance \eqref{eq9.14}-\eqref{eq9.15} in the text below). Now, in the first case, the strategy is to see equation \eqref{eq1.0} as a nondegenerate equation in $Q$, since $u$ does not vanish; while in the second case it would be desirable to apply the technique of {\it intrinsic scaling} developed by DiBenedetto  for the parabolic case \cite{DiB1993}, exploiting the original idea of the level-set method proposed by DeGiorgi \cite{DG}. The implementation of this technique differs considerably depending on how the equation behaves. Therefore, for the doubly-nonlinear equation \eqref{eq1.0}, the procedure encumbers dramatically when the two behaviors, singular and degenerate, are combined. The precise aim of our work here is to step through this gap and show that for a range of exponents in the mixed singular-degenerate or degenerate-singular regime, it is still possible to recover the desired regularity.

\noindent \subsection{Main Result}  In this paper, we are concerned with doubly nonlinear parabolic equations whose structure is modeled after the prototype \eqref{eq1.0}. Let indeed $\Omega$ be a domain in $\mathbb{R}^{N}$, $T>0$, $\Omega_{T}:= \Omega \times (0, T)$. 
We study nonnegative (sub) super-solutions to the equation
\begin{equation}\label{eq1.1}
u_{t}-\textrm{div}\mathbf{A}(x, t, u, D u)=0, \quad \quad  (x, t)\in \Omega_{T}.
\end{equation}
Throughout the paper we suppose that the functions $\mathbf{A}=(A_1, \cdots, A_N):\Omega_{T}\times \mathbb{R}_+\times \mathbb{R}^{N} \rightarrow \mathbb{R}^{N}$ are Carath\'eodory, i.e. such that
$\mathbf{A}(\cdot, \cdot, u, \xi)$ are Lebesgue measurable for all  $u\in \mathbb{R}_+$, $\xi \in \mathbb{R}^{N}$,
and $\mathbf{A}(x, t, \cdot,  \cdot)$ are continuous for almost all $(x, t)\in \Omega_{T}$.
We also assume that the following structure conditions are satisfied
\begin{equation}\label{eq1.2}
\begin{cases}
\mathbf{A}(x, t, u, D u) D u^{m^-} \geqslant  K_{1}\,u^{(m-m^-)(p-1)}|D u^{m^-}|^{p},\\
|\mathbf{A}(x, t, u,  D u)| \leqslant  K_{2}\,u^{(m-m^-)(p-1)}|D u^{m^-}|^{p-1},
\end{cases}
\end{equation}
where $m>0$, $p>1$, $K_{1}$, $K_{2}$ are positive constants and $m^-:=\min(1,m)$.

\begin{definition}
We say that a function $u$ is a  nonnegative, locally bounded, local weak  (sub) super-solution to \eqref{eq1.1} if 
\begin{equation*}
0 \leqslant \, u \in  C_{\textrm{loc}}(0, T; L^{1+m^-}_{\textrm{loc}}(\Omega)),\,\qquad u^{m^-} \in L_{\textrm{loc}}^{p}(0, T; W_{\textrm{loc}}^{1, p}(\Omega)),\,\qquad u\in L^\infty_{loc}(\Omega_T),
\end{equation*}
and for any compact set $E \subset \Omega$
and every subinterval $[t_{1}, t_{2}]\subset (0, T]$ there holds
\begin{equation}\label{eq1.3}
\int\limits_{E}u \zeta\, dx \bigg|^{t_{2}}_{t_{1}} + \int\limits^{t_{2}}\limits_{t_{1}}\int\limits_{E}
\{-u\zeta_{\tau}+ \mathbf{A}(x, \tau, u, D u) \,D \zeta\}\, dx d\tau \leqslant (\geqslant)\, 0,
\end{equation}
 for any testing functions $\zeta \in W^{1,1+\frac{1}{m^-}}(0, T; L^{1+\frac{1}{m^-}}(E))\cap L^{p}(0, T; W_{0}^{1, p}(E))$, $\zeta \geqslant 0$.
\end{definition}
\begin{remark} In order to simplify the presentation, we have decided to give a definition of solution that already encodes the nonnegativity of $u$ and its local boundedness, besides a certain integrability of the gradient of small powers of $u$. For these topics, we refer to \cite{CHSS} and references therein.

\end{remark}

\noindent Our main result reads as follows. Let the main degeneracy exponent be 
\[\lambda=m(p-1)\,.\]

\begin{theorem}\label{th1.1}
Let $u$ be a nonnegative, local weak solution to \eqref{eq1.1}-\eqref{eq1.2} in $\Omega_T$ and assume also that one of the following conditions holds
\begin{equation}\label{eq1.7}
\lambda<1 \quad \text{and} \quad p+N(\lambda-1)>0,
\end{equation}
or
\begin{equation}\label{eq1.5}
\lambda>1\quad \text{and}\quad
\frac{2N}{N+1}<p<2,
\end{equation}
or
\begin{equation}\label{eq1.6}
\lambda>1\quad \text{and}\quad
p>2,
\end{equation}
then $u$ is locally H\"{o}lder continuous in $\Omega_T.$
\end{theorem}

%%%%%%%%%%%%%%%%%%%%%%%%%%%%%%%%%%%%%%%%%%%%%%%%%%%%%%%%%%%%%%%%%%%%%%%%%%%%%%%%%%%%%%%%%%%%%%%%%%%%%%%%%%%%%%%%%%%%%%%%%%

\subsection{Novelty and Significance} Referring to the subsection {\it Regularity Theory} of the Introduction, we see that our results are new in the singular-degenerate case
\[\lambda=m(p-1)<1 \qquad \text{with} \qquad p>2,\]
and degenerate-singular case
\[\lambda=m(p-1)>1\qquad  \text{with} \qquad p<2\,.\] \noindent In the first case, as Theorem \ref{th1.1} shows, we are able to show that nonnegative weak solutions are H\"older continuous through the use of the classic doubly nonlinear $L^1$-$L^1$ Harnack-inequality (see Theorem \ref{lem3.2} and Lemma \ref{lem4.1}), if the doubly-nonlinear Barenblatt number $p+N(\lambda-1)$ is positive. 

\vskip0.1cm \noindent In the second case, we show that, since $p<2$, a particular $L^1$-$L^1$ Harnack inequality holds true for the difference $(\sup_Q u-u)$ (see Theorem \ref{lem3.2}) and, by making use of it in a similar fashion (see Lemma \ref{lem6.1}),  we are able to prolonge the time information and achieve the desired regularity, but now only for a range of exponents tied to the condition $1<p<2$, and therefore linked to the positivity of the $p$-Laplacian Barenblatt number $p+N(p-2)$. Moreover, to the best of our knowledge, Theorem~\ref{lem3.2}
has not appeared previously in the literature and represents an independent
result of interest and novelty. 

\vskip0.1cm \noindent In Figure~\ref{fig_1}, we illustrate the ranges of the exponents $p$ and $m$ in the $(p,m)$-plane for which the H\"older continuity of solutions to \eqref{eq1.0} has been established in the existing literature (hatched region) and in the present paper (gray region). The purpose of this figure is to highlight the novelty of our results. In particular, it shows that the degenerate--singular and singular--degenerate cases have not been covered in the literature, whereas these ranges are treated in the present paper up to the Barenblatt number.
%
% \begin{figure}[H] \label{figure}
% \subfloat[\textcolor{blue}{In light gray it is represented the range of exponents $p,m$ such that solutions to \eqref{eq1.0} are locally H\"older continuous ($N=3$).}]{
% \includegraphics[width=0.47\linewidth]{figs/(m,p)-existing.pdf}
% }
% \hspace{10mm}
% \subfloat[\textcolor{blue}{In dark gray, the range of exponents that we cover with our result ($N=3$).}]{
% \includegraphics[width=0.47\linewidth]{figs/(m,p)-new.pdf}
% }
% \caption{\textcolor{blue}{Illustration of the novelty of the present results.}}
% \label{fig_1}
% \end{figure}
%
% \begin{figure}[H]
% \subfloat[Our results are in green]{
% \includegraphics[width=0.47\linewidth]{figs/(p,q)-existing.pdf}
% }
% \hspace{10mm}
% \subfloat[Existing results]{
% \includegraphics[width=0.47\linewidth]{figs/(p,q)-new.pdf}
% }
% \caption{(p,q)-settings}
% \label{fig_2}
% \end{figure}
%
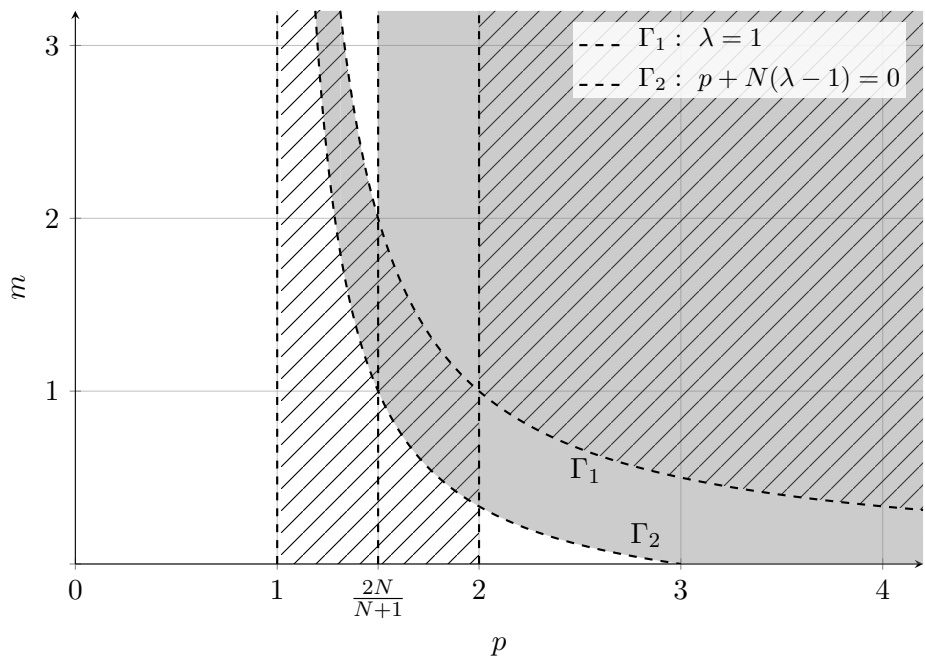
\begin{figure}[H]
\centering
\newcommand{\HatchDistance}{8pt}   % <-- hatching spacing (density control)
\newcommand{\HatchThickness}{0.4pt}

\pgfdeclarepatternformonly[\HatchDistance,\HatchThickness]
{my45hatch}{\pgfpoint{0pt}{0pt}}{\pgfpoint{\HatchDistance}{\HatchDistance}}{\pgfpoint{\HatchDistance}{\HatchDistance}}%
{
	\pgfsetlinewidth{\HatchThickness}
	\pgfpathmoveto{\pgfpoint{0pt}{0pt}}
	\pgfpathlineto{\pgfpoint{\HatchDistance}{\HatchDistance}}
	\pgfusepath{stroke}
}

% Custom gray color (#808080)
\definecolor{mygray}{HTML}{808080}

% Helpful constants (for clean domain splits)
% Top boundary: m = 3.2
% curve m = 1/(p-1) hits top at p = 1 + 1/3.2 = 1.3125
% gline m = (3-p)/(3(p-1)) hits top at p = 12.6/10.6 ? 1.188679245
\def\pCurveTop{1.3125}
\def\pGlineTop{1.188679245}

\begin{tikzpicture}
\begin{axis}[
% Use \linewidth so the picture auto-fits inside subfigures/minipages
width=0.80\linewidth,
height=0.56\linewidth,
xmin=0, xmax=4.2,
ymin=0, ymax=3.2,
% Remove top/right borders (no box frame)
axis lines=left,
xlabel={$p$},
ylabel={$m$},
% Ticks: add special tick at p=1.5 with label 2N/(N+1)
xtick={0,1,1.5,2,3,4},
xticklabels={0,1,$\frac{2N}{N+1}$,2,3,4},
% Remove duplicated "0" at the origin by hiding the y=0 label
ytick={0,1,2,3},
yticklabels={,1,2,3},
% Light grid at major ticks
grid=both,
major grid style={line width=0.2pt, draw=black!25},
minor tick num=0,
clip=true,
unbounded coords=discard,
restrict y to domain=0:3.2,
% -------- Legend as a text box (NO line samples) --------
legend entries={},
legend pos=north east,
legend cell align=left,
legend style={
	draw=none,
	fill=white,
	fill opacity=0.85,
	text opacity=1,
	font=\small
},
]

% Boundary: lambda=1 (i.e., m(p-1)=1)  <=>  m = 1/(p-1)
\addplot[
name path=curve,
domain=1.02:4.2, samples=120,
draw=none
] {1/(x-1)};

% Top and bottom of the plotting window
\addplot[name path=top, domain=0:4.2, samples=2, draw=none] {3.2};
\addplot[name path=bottom, domain=0:4.2, samples=2, draw=none] {0};

% Boundary: p + 3(m(p-1)-1)=0  <=>  m = (3-p)/(3(p-1))
\addplot[
name path=gline,
domain=1.02:3, samples=120,
draw=none
] {(3-x)/(3*(x-1))};

% lambda=1: start exactly where the curve meets m=3.2
\addplot[dashed, thick, domain=\pCurveTop:4.2, samples=200] {1/(x-1)};

% p+3(m(p-1)-1)=0: start exactly where it meets m=3.2
\addplot[dashed, thick, domain=\pGlineTop:3, samples=200] {(3-x)/(3*(x-1))};

% Vertical interval boundaries (dashed)
\addplot[dashed, thick] coordinates {(1,0) (1,3.2)};
\addplot[dashed, thick] coordinates {(1.5,0) (1.5,3.2)};
\addplot[dashed, thick] coordinates {(2,0) (2,3.2)};

% ------------------------------------------------------------
% REGION 1 (hatched, 45 degrees)
% ------------------------------------------------------------

% (i) m(p-1)>1, p>2, m>0  =>  m > 1/(p-1),  p in (2,4.2]
\addplot[
draw=none,
pattern=my45hatch,
pattern color=black
] fill between[
of=curve and top,
soft clip={domain=2:4.2}
];

% (ii) m(p-1)<1, 1<p<2, m>0  =>  0 < m < 1/(p-1)
% The curve exceeds m=3.2 for p < 1.3125, so we split the fill.

% (ii-a) 1<p<1.3125: fill the entire strip 0..3.2
\addplot[
draw=none,
pattern=my45hatch,
pattern color=black
] fill between[
of=top and bottom,
soft clip={domain=1.02:\pCurveTop}
];

% (ii-b) 1.3125<p<2: fill below the curve (0..curve)
\addplot[
draw=none,
pattern=my45hatch,
pattern color=black
] fill between[
of=curve and bottom,
soft clip={domain=\pCurveTop:2}
];

% ------------------------------------------------------------
% REGION 2 (transparent gray overlay), fill opacity = 0.4
% ------------------------------------------------------------

% (a) m(p-1)>1 and 1.5<p<2  =>  m > 1/(p-1)
\addplot[
fill=mygray, fill opacity=0.4,
draw=none
] fill between[
of=curve and top,
soft clip={domain=1.5:2}
];

% (b) m(p-1)<1 and p+3(m(p-1)-1)>0 and m>0
%     <=> m < 1/(p-1) AND m > (3-p)/(3(p-1)) AND m>0.

% (b0) For \pGlineTop < p < \pCurveTop: upper bound is the top (3.2), not the curve.
%      Fill: gline < m < top
\addplot[
fill=mygray, fill opacity=0.4,
draw=none
] fill between[
of=top and gline,
soft clip={domain=\pGlineTop:\pCurveTop}
];

% (b1) For \pCurveTop < p < 3: upper bound is the curve.
%      Fill: gline < m < curve
\addplot[
fill=mygray, fill opacity=0.4,
draw=none
] fill between[
of=curve and gline,
soft clip={domain=\pCurveTop:3}
];

% (b2) For 3 < p < 4.2: gline <= 0, so the condition m>0 dominates.
%      Fill: 0 < m < curve
\addplot[
fill=mygray, fill opacity=0.4,
draw=none
] fill between[
of=curve and bottom,
soft clip={domain=3:4.2}
];

% (c) m(p-1)>1 and p>2 (explicit gray overlay on the right part)
\addplot[
fill=mygray, fill opacity=0.4,
draw=none
] fill between[
of=curve and top,
soft clip={domain=2:4.2}
];

% ------------------------------------------------------------
% Labels on boundary lines (Gamma labels on the plot)
% ------------------------------------------------------------

\node[anchor=north west] at (axis cs:2.4,{0.67}) {$\Gamma_1$};
\node[anchor=south west] at (axis cs:2.7,{(3-2.7)/(3*(2.7-1))-0.02}) {$\Gamma_2$};

% ------------------------------------------------------------
% Legend text (no line samples)
% ------------------------------------------------------------

\node[
anchor=north east,
fill=white, fill opacity=0.85, text opacity=1,
inner sep=2pt,
font=\small,
align=left
] at (rel axis cs:0.98,0.98) {%
	\tikz{\draw[dashed,thick] (0,0)--(0.2,0);} \ $\Gamma_1:\ \lambda=1$\\
	\tikz{\draw[dashed,thick] (0,0)--(0.2,0);} \ $\Gamma_2:\ p+N(\lambda-1)=0$%
};

\end{axis}
\end{tikzpicture}

\caption{Illustration of the novelty of the present results. The hatched region represents the range of exponents $(p,m)$ for which the local H\"older continuity of solutions to \eqref{eq1.0} was established in the existing literature.  The solid gray region represents the range covered by the present work (for the case $N=3$).}
\label{fig_1}
\end{figure}

\subsection{Structure of the paper} In Section \ref{prelim}, we set the main notation and collect all the basic ingredients of our proofs: Energy estimates, De Giorgi-type Lemmas, Expansion of Positivity  and some other integral estimates. For most of the auxiliary results in Section~\ref{prelim} (except the ones related to estimates leading to integral Harnack-type inequalities), we do not provide the proofs but instead refer the reader to the corresponding results for the equation \eqref{dnleq3}.
These results are applicable here in essentially the same form after the appropriate change of variables. In Section~\ref{Main components}, we present and prove $L^1$-$L^1$ inequalities, the main components of the proof of the reduction of oscillation given by Proposition~\ref{pr1.1}. These Harnack-type estimates are derived separately in the case $\lambda<1$ and ($\lambda>1$ and $ p<2$). In Section \ref{sec-red-osc}, we show through Proposition \ref{pr1.1} that if the local supremum of $u$ is close enough to the oscillation, then we have a reduction of the oscillation (referring to alternative 2 of the paragraph {\bf {\it The issue of H\"older Continuity}} of the Introduction). Finally, in Section 5 we prove the main Theorem \eqref{th1.1}, by completing the alternative (alternative 1 of the paragraph {\bf {\it The issue of H\"older Continuity}} of the Introduction).

\section{Notations and Preliminary Results}\label{prelim}

\subsection{Notations}

In the notation of equation \eqref{eq1.1}, \eqref{eq1.2}, we refer to
$N$, $m$, $p$, $K_1$, and $K_2$ as the {\it data} of the equation. We also define

\begin{equation*}
\lambda :=m(p-1) \qquad  \mathrm{and} \qquad m^-:=\min(1,m).
%\\M&:=\sup\limits_{\Omega_T}u.
\end{equation*}

We use the symbol $\gamma$ to denote a generic positive constant that can be
determined a priori solely in terms of the {\it data} (that is to say $N,m,p,K_1$ and $K_2$). The value of $\gamma$ may change from line to line.

\medskip
\noindent For a point $(x_0,t_0)\in\Omega_T$ and parameters $r,\eta>0$, we define the backward
parabolic cylinder
\begin{equation*}
Q_{r,\eta}(x_0,t_0):=B_r(x_0)\times(t_0-\eta,t_0),\qquad B_r(x_0):=\{x\in\mathbb{R}^N:|x-x_0|< r\}.
\end{equation*}
We will also frequently use the intrinsic scaling:
\begin{equation*}
\theta(k,r):=k^{1-\lambda}r^p, \qquad k,r>0.
\end{equation*}
%\begin{equation*}
%\begin{aligned}
%B_r(x_0)&:=\{x\in\mathbb{R}^N:|x-x_0|< r\},
%\\(x_0,t_0)+Q(r,\eta)&:=B_r(x_0)\times(t_0-\eta,t_0),
%\\ (x_0,t_0)+Q_r(\theta)&:=(x_0,t_0)+Q(r,\theta r^p)
%\end{aligned}
%\end{equation*}
Finally, for any measurable set $E\subset\mathbb{R}^{N}$ with positive measure, we denote by
\begin{equation*}
\fint\limits_E u\,dx:=\frac{1}{|E|}\int\limits_E u\,dx
\end{equation*}
the integral average of $u$ over $E$.

\subsection{Auxiliary Lemma}

Define the function
$$g_{\pm}(u^{m}, k^{m}):=\int\limits^{(u^{m}-k^{m})_{\pm}}_0\big(k^{m}\pm s\big)^{\frac{1}{m}-1}\,s\,ds,\quad m>0.$$
The following lemma can be extracted from \cite[Lemma~3.2]{BogDuzGiaLiaSch}.

\begin{lemma}\label{lem2.1}
There exists a constant $\gamma>0$, depending only on $m$, such that
\begin{equation*}
\frac{1}{\gamma}\,(k^m+u^m)^{\frac{1}{m}-1}(u^m-k^m)^2_{\pm}\leqslant g_{\pm}(u^m, k^m)\leqslant \gamma\, (k^m+u^m)^{\frac{1}{m}-1}(u^m-k^m)^2_{\pm}.
\end{equation*}
\end{lemma}

% \subsection{Mollified weak formulation} \label{approx}
% It is technically convenient to have a formulation of a nonnegative, local, weak (sub) super-solution that involves $u_{t}$. Let $0 \leq \rho(x)\in C_{0}^{\infty}(\mathbb{R}^{N})$ be a standard mollifier vanishing outside the unit ball and normalized in $\mathbb{R}^N$. Let us define the rescaled functions
% \[ \rho_{h}(x):= h^{-N}\rho\left(x/h\right), \quad
% u_{h}(x, t):= h^{-1}\int\limits_{t}\limits^{t+h}\int\limits_{\mathbb{R}^{N}}u(y, \tau)\rho_{h}(x-y)\,dy d\tau .\]
% \noindent We fix $t \in (0, T)$ and let $h>0$ be so small that $0<t<t+h<T$.
% In \eqref{eq1.3} we take $t_{1}=t$, $t_{2}=t+h$ and replace $\zeta$ by $\int\limits_{\mathbb{R}^{N}}\zeta(y, t)\rho_{h}(x-y)\,dy$. Dividing by $h$, since the testing function independent of  $\tau$, we obtain that the function $u$ must satisfy
% \begin{equation}\label{eq1.4}
% \int\limits_{E\times \{t\}} \left(\frac{\partial [u]_{h}}{\partial t}\, \zeta+[\mathbf{A}(x, t, u, D u )]_{h} D \zeta\right)dx
% \leqslant (\geqslant)\,0,
% \end{equation}
% for all $t \in (0, T-h)$ and for all  $\zeta \in W^{1,p}_{0}(E)$, $\zeta \geqslant 0$.

\subsection{Local Energy Estimates}

%For the equation \eqref{dnleq3} (in $(p,q)$-settings for the signed solutions) can be found in \cite[Proposition 3.1]{BogDuzLia}.}

\begin{lemma}\label{lem2.2}
Let $u$ be a nonnegative, local weak subsolution (supersolution) to \eqref{eq1.1}-\eqref{eq1.2} with $p>1$ and $m>0$. Then there exists a constant $\gamma>0$, depending only on the data, such that for every cylinder $Q_{r,\eta}(y,\tau)\subset\Omega_{T}$, every $k\in\mathbb{R}_{+}$, and every piecewise smooth cutoff function $\zeta$ vanishing on $\partial B_{r}(y)$ and satisfying $0\leqslant \zeta \leqslant 1$, the following inequality holds:
\begin{equation}\label{eq2.1}
\begin{aligned}
&\sup\limits_{\tau-\eta\leqslant t \leqslant \tau}\int\limits_{B_r(y)}g_{\pm}(u^{m^-}, k^{m^-})\,\zeta^{p}\,dx+\gamma^{-1} \iint\limits_{Q_{r,\eta}(y, \tau)}u^{(m-m^-)(p-1)}|D (u^{m^-}-k^{m^-})_{\pm}|^{p}\zeta^{p}\,dx\,dt
\\&\leqslant\int\limits_{B_r(y)\times\{\tau-\eta\}}g_{\pm}(u^{m^-}, k^{m^-})\,\,\zeta^{p}\,dx+\gamma\iint\limits_{Q_{r,\eta}(y, \tau)} g_{\pm}(u^{m^-}, k^{m^-})\,|\zeta_t|\,dx dt
\\&\quad+\gamma \iint\limits_{Q_{r,\eta}(y, \tau)}u^{(m-m^-)(p-1)}(u^{m^-}-k^{m^-})_{\pm}^{p}|D \zeta|^{p}\,dx\,dt.
\end{aligned}
\end{equation}
\end{lemma}
\begin{proof}
Test \eqref{eq1.3} by $\pm (u^{m^-}-k^{m^-})_{\pm}\zeta^{p}$ and integrate over $B_{r}(y)\times(\tau-\eta,t)$ for $t\in(\tau-\eta,\tau)$. The use of such a test function is justified, modulus a standard averaging process (see \cite{KL} for the original idea and \cite{CVV} for the details). Using condition \eqref{eq1.2} and  Young's inequality, we obtain the required \eqref{eq2.1}.
\end{proof}

\subsection{De Giorgi Type Lemmas}

In this subsection, we present several well-known De Giorgi type lemmas that are used in the proof of Proposition~\ref{pr1.1}. The proofs of these lemmas are mostly standard and follow classical arguments. Therefore, we do not repeat them here and instead refer the reader to existing results in the literature.

 \noindent We introduce the cylinder $Q_{2\,r, 2\eta}(y, \tau)\subset Q_{2r, 2r}(y, \tau) \subset \Omega_T$, and define the values $\mu^{\pm}$ and $\omega$ by
\begin{equation*}
\mu^+\geqslant \sup\limits_{Q_{r, \eta}(y, \tau)} u,
\qquad \mu^-\leqslant \inf\limits_{Q_{r, \eta}(y, \tau)} u,
\qquad \omega\geqslant \mu^+-\mu^-.
\end{equation*}

\noindent The following result for equation~\eqref{dnleq3} can be found in \cite[Lemma~6.1]{BogDuzGiaLiaSch} and \cite[Lemma~3.2]{Hen2022}. The corresponding proof for the anisotropic version of \eqref{eq1.1}, \eqref{eq1.2} is given in \cite[Lemma~2.4]{CiaHenSavSkr2025-1}.

\begin{lemma}\label{lem2.5}
Let $u$ be a nonnegative, weak supersolution to \eqref{eq1.1}-\eqref{eq1.2}. For some $b\in(0,1]$ and $\xi\in(0,1)$, consider the cylinder $Q_{r,b\,\theta(\xi\omega,r)}(y, \tau) \subset \Omega_T$.
Then there exists a constant $\nu_{-}\in(0,1)$, depending only on the data and on $b$, such that if
\begin{equation}\label{eq2.5}
|Q_{r,b\,\theta(\xi\omega,r)}(y, \tau)\cap\{u\leqslant \xi\,\omega\}|\leqslant \nu_- |Q_{r,b\,\theta(\xi\omega,r)}(y, \tau)|,
\end{equation}
then 
\begin{equation}\label{eq2.6}
u(x,t)\geqslant \frac{1}{2}\, \xi\,\omega,\quad (x,t)\in Q_{\frac{1}{2}r,\frac{1}{2}b\,\theta(\xi\omega,r)}(y, \tau).
\end{equation}
Moreover,
\begin{equation*}
\nu_-=\frac{1}{\gamma}\, b^{\frac{N}{p}}
\end{equation*}
with some constant $\gamma>0$ depending only on the data.
\end{lemma}

\noindent The following lemma is a variant of a De Giorgi type result adapted to our setting. For the proof in the anisotropic case of equations \eqref{eq1.1}, \eqref{eq1.2}, we refer the reader to \cite[Lemma~2.5]{CiaHenSavSkr2025-1}.

\begin{lemma}\label{lem2.6}
Let $u$ be a nonnegative, weak subsolution to \eqref{eq1.1}-\eqref{eq1.2}. For some $b\in(0,1]$, consider the cylinder $Q_{r,b\,\theta(\omega,r)}(y, \tau) \subset \Omega_T$.
Then there exists a constant $\nu_{+}\in(0,1)$, depending only on the data and on $b$, such that if
\begin{equation}\label{eq2.12}
|Q_{r,b\,\theta(\omega,r)}(y, \tau)\cap\{u\geqslant \mu^{+}-\frac{1}{4}\omega\}|\leqslant \nu_+ |Q_{r, b\,\theta(\omega,r)}(y, \tau)|,
\end{equation}
then
\begin{equation}\label{eq2.13}
u(x, t)\leqslant \mu^{+}- \frac{1}{8} \omega,\quad (x,t)\in Q_{\frac{1}{2}r,\frac{1}{2}b\,\theta(\omega,r)}(y, \tau),
\end{equation}
provided that
\begin{equation}\label{eq2.14}
\frac{1}{2}\,\omega\leqslant \mu^+\leqslant 2\,\omega.
\end{equation}
Moreover,
\begin{equation*}
\nu_+:=\frac{1}{\gamma}\,b^{\frac{N}{p}}.
\end{equation*}
with some constant $\gamma>0$ depending only on the data.
\end{lemma}

%\subsection{De Giorgi Type Lemmas Involving the Initial Data}

The following result is a De Giorgi type lemma involving the initial data and is used in the proof of Proposition~\ref{pr1.1} in the doubly degenerate case \eqref{eq1.6}. A related result for equation~\eqref{dnleq3} can be found in \cite[Lemma~3.5]{BDLS23}.

\begin{lemma}\label{lem2.7}
Let $u$ be a nonnegative, local weak subsolution to \eqref{eq1.1}-\eqref{eq1.2} in $\Omega_T$, and assume that $\lambda>1$ and $p>2$. Fix $\xi_0\in (0, 1)$ and $b\geqslant 1$, and suppose in addition that
\begin{equation}\label{eq2.17}
u\big(x, \tau-b\,\theta(\omega,r)\big)\leqslant \mu^+- \xi_0\,\omega,\quad x\in B_r(y).
\end{equation}
Then there exists a constant $\xi\in(0,\xi_0)$, depending only on the data and on $b$, such that 
\begin{equation}\label{eq2.18}
u(x, t)\leqslant \mu^+- \frac{1}{4}\xi \omega,\quad (x, t)\in Q_{\frac{1}{2}r,b\,\theta(\omega,r)}(y, \tau),
\end{equation}
provided that
\begin{equation}\label{eq2.20}
\frac{1}{2}\,\omega \leqslant \mu^+\leqslant 2\,\omega.
\end{equation}
\end{lemma}

% We do not use this lemma anywhere!
%\begin{lemma}\label{lem2.9}
%Let $u$ be a non-negative local weak sub-solution to \eqref{eq1.1}, \eqref{eq1.2} in $\Omega_T$, $\lambda>1$ and $p<2$. Assume that with some $\xi\in (0, 1)$ there holds
%\begin{equation}\label{eq2.29}
%u(x, \tau)\leqslant \mu^+- \xi\,\omega,\quad x\in B_r(y),
%\end{equation}
%then then there exists number $\nu_0 \in (0, 1)$ depending only on the data  and $\xi$ such that 
%\begin{equation}\label{eq2.30}
%u(x, t)\leqslant \mu^+ - \frac{1}{4}\xi \omega,\quad (x, t)\in B_{\frac{1}{4}r}(y)\times(\tau, \tau+\eta),\quad \eta:=\nu_0\,r^{p}\,\omega^{1-\lambda},
%\end{equation}
%provided that
%\begin{equation}\label{eq2.31}
%\frac{1}{2}\omega \leqslant \mu^+ \leqslant 2 \omega.
%\end{equation}
%\end{lemma}

\subsection{Expansion of Positivity}

In this section, we present several results on expansion of positivity, which form one of the main ingredients in the proof of Proposition~\ref{pr1.1}.

We begin with the case $\lambda<1$. A corresponding result for equation \eqref{dnleq3} can be found in \cite[Proposition~4.3]{Hen2022}, while the anisotropic version of \eqref{eq1.1}-\eqref{eq1.2} is treated in \cite[Theorem~3.4]{CiaHenSavSkr2025-1}.

\begin{lemma}\label{lem3.3}
Let $u$ be a nonnegative, local weak supersolution to \eqref{eq1.1}-\eqref{eq1.2}, and assume that $\lambda<1$. Suppose that for some $0<k\leqslant M$, $r>0$, and $\alpha\in(0,1)$, the following estimate holds
\begin{equation}\label{eq3.5}
|B_r(y)\cap\{u(\cdot, s) \geqslant k\}| \geqslant \alpha\, |B_r(y)|,
\end{equation}
Then there exist constants $\sigma, \varepsilon, \delta \in (0,1)$
depending only on data and on $\alpha$, such that
\begin{equation}\label{eq3.6}
u(x, t)\geqslant\sigma k,\quad x\in B_{2r}(y),
\end{equation}
for all $t$ in the interval
\begin{equation}\label{eq3.7}
s+\frac{1}{2}(1-\varepsilon)\delta\,\theta(k, r)\leqslant t\leqslant s+\frac{1}{2}\delta\,\theta(k, r).
\end{equation}
%\textcolor{magenta}{$0<b_1<b_2$  (by the proof of exp of positivity $b_1=\beta b_2, \beta<1$, $b_2=\frac{1}{2}\frac{\delta}{\sigma_0^{1-\lambda}}, \sigma_0=\frac{1}{C}$, which means that we can not guarantee that $b_2<1$, but we can guarantee that $b_1\sigma_0^{1-\lambda}<\frac{1}{2}$, which is essental for us for the proof of Proposition \ref{pr1.1}), $C>1$} 
\end{lemma}

We next consider the case $1<p<2$. We refer the reader to \cite[Lemma~3.4]{NaS22} for the corresponding result in the setting of equation~\eqref{dnleq3} for signed solutions.

\begin{lemma}\label{new_exp_lem}
Let $u$ be a nonnegative, local weak subsolution (supersolution) to \eqref{eq1.1}-\eqref{eq1.2} in $\Omega_T$, and let $m>0$, $1<p<2$. Suppose that for some $0<c<C$ and $\alpha\in(0,1)$, there holds
$$c\omega\leqslant \mu^+ \leqslant C\omega$$
and that for some $0<a < \frac12c$, 
\begin{equation*}
\Big|B_{r}(y)\cap \big\{\mu^{+}-u(\cdot, s)\geqslant a\omega\big\}\Big|\geqslant \alpha |B_r(y)|.
\end{equation*}
Then there exist constants $\sigma,\delta\in(0,1)$ depending only on the data and on $C$,  $c$, and $\alpha$, such that 
\begin{equation*}
\mu^{+}-u\geqslant \sigma\, a\, \omega
\quad x\in B_{2r}(y)\times\{s+\delta\,a^{2-p}\theta(\omega,r)\},
\end{equation*}
provided that $B_{2r}(y)\times\left(s,s+\delta\,a^{2-p}\theta(\omega,r)\right)\in\Omega_T$.
\end{lemma}

Finally, we treat the case $\lambda>1$. A corresponding result for equation \eqref{dnleq3} is proved in \cite[Proposition~4.3]{Hen2022}, while the anisotropic version of \eqref{eq1.1}, \eqref{eq1.2} can be found in \cite[Theorem~3.1]{CiaHenSavSkr2025-1}.

\begin{lemma}\label{deg_exp_lem}
Let $u$ be a nonnegative, local weak supersolution of \eqref{eq1.1}-\eqref{eq1.2} in $\Omega_T$, and assume that $\lambda>1$. Suppose that for some constants $k, r>0$ and $\alpha \in (0, 1)$, the following estimate holds
\begin{equation*}
|B_r(y)\cap\{u(\cdot, s)\geqslant k\}|\geqslant \alpha |B_r(y)|,
\end{equation*}
Then there exist constants  $\delta>1$ and $\sigma\in(0,1)$ depending only on the data and on $\alpha$, such that
\begin{equation*}
u(x, t)\geqslant \sigma\,k,\quad x\in B_{2r}(y),
\end{equation*}
for all $t$ in the interval
\begin{equation*}
s+ \frac{1}{2} \delta\,\theta(k,r) \leqslant t\leqslant s+ \delta \,\theta(k,r).
\end{equation*}
\end{lemma}

\subsection{Other integral estimates}

\noindent We introduce and prove the following two auxiliary lemmas, which will be used to establish the $L^{1}_{\mathrm{loc}}-L^{1}_{\mathrm{loc}}$ Harnack inequalities (another novelty related to this work) presented in Section~\ref{Main components}.

\begin{lemma}\label{lem2.3}
Let $u$ be a nonnegative, local weak supersolution of \eqref{eq1.1}-\eqref{eq1.2} in $\Omega_T$. Fix parameters $0<\alpha<1$ and $1<\beta<1+\frac{1}{m^{-}}$. Then there exists a constant $\gamma>0$, depending only on the data and on $\beta$, such that for every $\varepsilon>0$, every cylinder $B_r(y)\times(s,t)\subset\Omega_T$, and every piecewise smooth cutoff function $\zeta$ vanishing on $\partial B_r(y)$ and satisfying $0\leqslant\zeta\leqslant 1$, the following inequality holds:
\begin{equation}\label{eq2.2}
\begin{aligned}
&\int\limits_s^t\int\limits_{B_r(y)}\Big(\frac{\tau-s}{t-s}\Big)^\alpha
\frac{u^{(m-m^-)(p-1)}}{(u^{m^-}+\varepsilon^{m^-})^\beta}|D u^{m^-}|^{p}\,\zeta^{p}\,dx d\tau
\\&\leqslant \gamma \int\limits_{B_r(y)\times\{t\}}(u+\varepsilon)^{1-m^-(\beta-1)}\,dx+
\gamma \int\limits_s^t\int\limits_{B_r(y)}(u+\varepsilon)^{\lambda-m^-(\beta-1)}\,|D \zeta|^{p}\,dx d\tau.
\end{aligned}
\end{equation}
\end{lemma}
\begin{proof}
Test \eqref{eq1.3} by $(u^{m^-}+\varepsilon^{m^-})^{1-\beta}\big(\frac{\tau-s}{t-s}\big)^\alpha \zeta^p$ with $\zeta\in C_0^\infty(B_r(y))$, $0\leqslant\zeta\leqslant1$ and integrate over $B_r(y)\times(s, t)$.
The use of such a test function is justified, modulus a standard
 averaging process, by making use of the alternate weak formulation of \cite{KL}. Using conditions \eqref{eq1.2} and  Young's inequality, we obtain
\begin{equation*}
\begin{aligned}
& K_1(\beta-1)\int\limits_s^t\int\limits_{B_r(y)}\Big(\frac{\tau-s}{t-s}\Big)^\alpha 
\frac{u^{(m-m^-)(p-1)}}{(u^{m^-}+\varepsilon^{m^-})^\beta}|D u^{m^-}|^{p}\,\zeta^{p}\,dx d\tau
\\&\leqslant \gamma \int\limits_{B_r(y)\times\{t\}}\int\limits_0^u\frac{ds}{(s^{m^-}+\varepsilon^{m^-})^{\beta-1}}\,\zeta^{p}\,dx
\\&\quad+\gamma\int\limits_s^t\int\limits_{B_{r}(y)}\Big(\frac{\tau-s}{t-s}\Big)^\alpha u^{(m-m^-)(p-1)}(u^{m^-}+\varepsilon^{m^-})^{p-\beta}|D\zeta|^{p}\,dx d\tau.
\end{aligned}
\end{equation*}
First, we estimate the term corresponding to the time derivative of the test function
\begin{equation*}
\begin{aligned}
\gamma \int\limits_{B_{r}(y)\times\{t\}}\int\limits_0^u\frac{ds}{(s^{m^-}+\varepsilon^{m^-})^{\beta-1}}\,\zeta^{p}\,dx
&\leqslant\gamma \int\limits_{B_{r}(y)\times\{t\}}\int\limits_0^u\frac{ds}{(s+\varepsilon)^{m^-(\beta-1)}}\,dx
\\&\leqslant \gamma(\beta)\int\limits_{B_r(y)}(u+\varepsilon)^{1-m^-(\beta-1)}\,dx.
\end{aligned}
\end{equation*}
The remaining terms, which involve the spatial derivatives of the cutoff function $\zeta$, can be estimated as follows
\begin{equation*}
\begin{aligned}
&\int\limits_s^t\int\limits_{B_r(y)}\Big(\frac{\tau-s}{t-s}\Big)^\alpha u^{(m-m^-)(p-1)}(u^{m^-}+\varepsilon^{m^-})^{p-\beta}|D\zeta|^{p}\,dx d\tau
\\&\leqslant
\int\limits_s^t\int\limits_{B_r(y)}(u^{m^-}+\varepsilon^{m^-})^{\frac{m}{m^-}(p-1)+1-\beta}|D \zeta|^{p}\,dx d\tau
\leqslant\int\limits_s^t\int\limits_{B_r(y)}(u+\varepsilon)^{\lambda-m^-(\beta-1)}|D\zeta|^{p}\,dx d\tau.
\end{aligned}
\end{equation*}
Combining the above estimates, we arrive at the required inequality \eqref{eq2.2}.
\end{proof}

\begin{lemma}\label{lem2.4}
Let $u$ be a nonnegative, local weak subsolution to \eqref{eq1.1}-\eqref{eq1.2} in $\Omega_T$. Fix parameters $0<\alpha<1$ and $1<\beta<2$, and consider a cylinder $B_r(y)\times(s,t)\subset\Omega_T$. Assume that
\begin{equation}\label{eq2.3}
u\leqslant \mu^+,\quad \text{in}\quad B_r(y)\times (s, t),
\end{equation}
Then there exists a constant $\gamma>0$, depending only on the data and on $\beta$, such that for every $\varepsilon>0$ and every piecewise smooth cutoff function $\zeta$ vanishing on $\partial B_r(y)$ and satisfying $0\leqslant\zeta\leqslant 1$, the following inequality holds:
\begin{equation}\label{eq2.4}
\begin{aligned}
&\int\limits_s^t\int\limits_{B_r(y)}\Big(\frac{\tau-s}{t-s}\Big)^\alpha\frac{u^{(m-1)(p-1)}}{(\mu^+-u+\varepsilon)^\beta}|D u|^{p}\,\zeta^{p}\,dx d\tau
\\&\leqslant\gamma \int\limits_{B_r(y)\times\{t\}}\int\limits^{\mu^+}_u\frac{d z}{(\mu^+-z +\varepsilon)^{\beta-1}}\,dx
+\gamma [\mu^+]^{\lambda-p+1}\int\limits_s^t\int\limits_{B_r(y)}(\mu^+-u+\varepsilon)^{p-\beta}\,|D \zeta|^{p}\,dx d\tau.
\end{aligned}
\end{equation}
\end{lemma}
\begin{proof}
Test \eqref{eq1.3} by $(\mu^+-u +\varepsilon)^{1-\beta}\big(\frac{\tau-s}{t-s}\big)^\alpha \zeta^{p}$ and integrate over $B_r(y)\times(s, t)$.  The use of such a test function is justified, modulus a standard
 averaging process, by making use of the alternate weak formulation \cite{KL}. Using conditions \eqref{eq1.2} and Young's inequality, we obtain
\begin{equation*}
\begin{aligned}
&K_1(\beta-1)\int\limits_s^t\int\limits_{B_r(y)}\Big(\frac{\tau-s}{t-s}\Big)^\alpha
\frac{u^{(m-1)(p-1)}}{(\mu^+-u +\varepsilon)^\beta}|D u|^{p}\,\zeta^{p}\,dx d\tau
\\&\leqslant
\gamma \int\limits_{B_r(y)\times\{t\}}(\mu^+-u +\varepsilon)^{2-\beta}\,\zeta^{p}\,dx+
\gamma  \int\limits_s^t\int\limits_{B_r(y)}u^{(m-1)(p-1)}(\mu^+-u+\varepsilon)^{p-\beta}\,|D \zeta|^{p}\,dx d\tau.
\end{aligned}
\end{equation*}
By representing the first term on the right-hand side of the previous inequality as an integral with a variable lower limit and applying condition \eqref{eq2.3} to the second term, we derive \eqref{eq2.4}.
\end{proof}
%%%%%%%%%%%%%%%%%%%%%%%%%%%%%%%%%%%%%%%%%%%%%%%%%%%%%%%%%%%%%%%%%%%%%%%%%%%%%%%%%%%%%%%%%%%%%%%%%%%%%%%%%%%%%%%%%%%%%%%%%%
\section{$L^{1}_{\mathrm{loc}}-L^{1}_{\mathrm{loc}}$ Harnack Inequalities}\label{Main components}

In this section, we establish the main components of the proof of Proposition~\ref{pr1.1}, the $L^{1}_{\mathrm{loc}}$-$L^{1}_{\mathrm{loc}}$ Harnack inequalities. The following theorem states the $L^{1}_{\mathrm{loc}}$-$L^{1}_{\mathrm{loc}}$ Harnack inequality in the case $\lambda<1$. Corresponding results for equation~\eqref{dnleq3} can be found in \cite[Proposition~7.1]{BogDuzGiaLiaSch}.

\begin{theorem}\label{lem3.1}
Let $u$ be a nonnegative, local weak solution to \eqref{eq1.1}-\eqref{eq1.2}, and assume that $\lambda<1$. Then there exists a constant $\gamma>0$, depending only on the data, such that for any cylinder $Q_{2r,t-s}(y,t)\subset\Omega_T$, the following estimate holds:
\begin{equation}\label{eq3.1}
\sup\limits_{s\leqslant \tau\leqslant t}\fint\limits_{B_{r}(y)}u(x, \tau)\, dx \leqslant \gamma
\inf\limits_{s\leqslant \tau \leqslant t}\fint\limits_{B_{2 r}(y)} u(x, \tau)\,dx +\gamma
\Big(\frac{t-s}{r^{p}}\Big)^{\frac{1}{1-\lambda}}.
\end{equation}
\end{theorem}
\begin{proof}
Without loss of generality, we set $s=0$.
For $j=0,1,2,\dots$, define
$$
r_j:= \sum\limits^{j}\limits_{l=1} \frac{r}{2^{l}},
\qquad
\bar{r}_j := \frac{r_j+ r_{j+1}}{2},
$$
and set
$$
B_j:=B_{r_j}(y),
\qquad
\bar B_j:=B_{\bar{r}_j}(y),\qquad B_j\subset\bar B_j\subset B_{j+1}.
$$
Let $\zeta_j\in C_0^1(\bar B_j)$ be a cutoff function satisfying
$$
\zeta_j\equiv1 \text{ in } B_j,
\qquad
0\leqslant\zeta_j\leqslant 1,
\qquad
|D_i\zeta_j|\leqslant \gamma\,2^{j}r^{-1},
\quad i=1,\dots,N.
$$
Fix $0<t_1<t_2<t$.
Testing \eqref{eq1.3} with $\zeta_j(x)$ and integrating over
$(t_1,t_2)\times\bar B_j$, we obtain
\begin{equation}\label{pre-recurs_ineq}
\fint\limits_{\bar B_j}u(x,t_2)\zeta_j\,dx
\leqslant
\fint\limits_{\bar B_j}u(x,t_1)\zeta_j\,dx
+
\frac{\gamma\,2^j}{r}
\int\limits_{t_1}^{t_2}\fint\limits_{\bar B_j}
u^{(m-m^-)(p-1)}
|D u^{m^-}|^{p-1}\,dx\,d\tau .
\end{equation}
Choose $t_1$ such that
$$
\fint\limits_{B_{2r}(y)} u(x, t_{1})\,dx= \inf\limits_{0<\tau<t} \fint\limits_{B_{2r}(y)\times\{\tau\}} u\,\,\,dx,
$$
Since $\bar B_j\subset B_{2 r}$, we obtain
$$
\fint\limits_{\bar B_j}u(x,t_1)\,\zeta_j\,dx
\leqslant \gamma
\inf\limits_{0<\tau<t} \fint\limits_{B_{2r}(y)\times\{\tau\}} u\,\,\,dx.
$$
Define $J_j:=\sup\limits\limits_{0<\tau<t}\,\,\fint\limits_{B_j}u(x,\tau)\,dx$, we arrive at
\begin{equation}\label{recurs_1}
J_j
\leqslant \gamma
\inf\limits_{0<\tau<t}\fint\limits_{B_{2r}(y)\times\{\tau\}}u\,\,\,dx
+
\frac{\gamma\,2^j}{r}
\int\limits_0^t\fint\limits_{\bar B_j}
u^{(m-m^-)(p-1)}
|D u^{m^-}|^{p-1}\,dx\,d\tau .
\end{equation}
In order to estimate the second term in the previous inequality, we fix  $\varepsilon =  \left(\frac{t}{r^{p}}\right)^{\frac{1}{1-\lambda}},$
and let $\alpha>0$ and $\beta>1$ be constants to be chosen later.  
Applying H\"older's inequality, we obtain
\begin{align}\label{eq5.4}
&\frac{\gamma\,2^j}{r}\int\limits_0^{t}\fint\limits_{\bar B_j}u^{(m-m^-)(p-1)}|D u^{m^-}|^{p-1}\,dx d\tau
\leqslant  \frac{\gamma\,2^j}{r} \Big(\int\limits_0^{t}\fint\limits_{ B_{j+1}}\Big(\frac{\tau}{t}\Big)^{\alpha}\frac{u^{(m-m^-)(p-1)}}{(u^{m^{-}}+\varepsilon^{m^{-}})^{\beta}}|D u^{m^-}|^{p}\,dx d\tau\Big)^{\frac{p-1}{p}}\notag\\
&\quad\times\Big(\int\limits_0^{t}\fint\limits_{B_{j+1}}
\Big(\frac{t}{\tau}\Big)^{\alpha(p-1)} u^{(m-m^{-})(p-1)}(u^{m^{-}}+\varepsilon^{m^{-}})^{\beta(p-1)}\,dx\,d\tau\Big)^{\frac{1}{p}}= \frac{\gamma 2^j}{r}  A_1^{\frac{p-1}{p}} A_2^{\frac{1}{p}}.
\end{align}
In order to estimate the right-hand side of \eqref{eq5.4}, we start by applying the elementary inequality for positive numbers: $u^{m^-(\frac{m}{m^-}-1)(p-1)} 
\leqslant(u^{m^-}+\varepsilon^{m^-})^{(\frac{m}{m^-}-1)(p-1)}, \varepsilon>0, $ and then take the supremum of the spatial average over the time interval $[0,t]$, to get
\begin{align*}
A_2&\leqslant\int\limits_0^{t}\fint\limits_{B_{j+1}}\Big(\frac{t}{\tau}\Big)^{\alpha(p-1)} (u^{m^{-}}+\varepsilon^{m^{-}})^{(\frac{m}{m^{-}}+\beta -1)(p-1)} dx d\tau
\\&\leqslant \sup\limits_{0<\tau<t}\fint\limits_{B_{j+1}}(u+\varepsilon)^{(m+m^{-}(\beta -1))(p-1)} dx \ \int\limits_0^t \Big(\frac{t}{\tau}\Big)^{\alpha(p-1)} d\tau.
\end{align*}
In what follows, we fix the values $\alpha$ and $\beta$. By choosing $\alpha$ such that $\alpha(p-1)<1$, the time integral on the right-hand side of the previous inequality is finite and can be estimated as
$\int_0^t \Big(\frac{t}{\tau}\Big)^{\alpha(p-1)}\, d\tau 
\leqslant \gamma\, t,$
where $\gamma$ is a positive constant depending only on $\alpha$ and $p$. Furthermore by choosing $\beta$ so that $(m+m^-(\beta-1))(p-1)<1$,  we can apply H\"older's inequality to the spatial integral. Hence,
\begin{equation}\label{eq5.5}
A_2\leqslant \gamma \,t \,\Big(\sup\limits_{0<\tau<t}\fint\limits_{ B_{j+1}\times \{\tau\}}(u+\varepsilon) \,dx\Big)^{(m+m^{-}(\beta -1))(p-1)}.
\end{equation}
To estimate $A_1$ of \eqref{eq5.4}, we apply Lemma~\ref{lem2.3}. 
Choosing $\beta$ such that $m^{-}(\beta-1)<1$,  using the simple estimate   $(u+\varepsilon)^{\lambda-1}\leqslant \varepsilon^{\lambda-1}$. Recalling the choice of $\varepsilon$
 and applying H\"older's inequality, we obtain
\begin{equation}\label{eq5.6}
\begin{aligned}
A_1&\leqslant \gamma\fint\limits_{B_{j+1}\times\{t\}}(u+\varepsilon)^{1-m^-(\beta-1)} dx+\gamma \int\limits_0^t\fint\limits_{B_{j+1}}
(u+\varepsilon)^{\lambda-m^-(\beta-1)}|D \zeta_{j}|^{p}\,dxd\tau\\
&\leqslant \gamma 2^{jp}\bigg[1+ \frac{t}{r^{p}\varepsilon^{1-\lambda}}\bigg]\sup\limits_{0<\tau<t}\fint\limits_{B_{j+1}\times\{\tau\}}
(u+\varepsilon)^{1-m^-(\beta-1)} dx\\&\leqslant \gamma 2^{jp}
\Big(\sup\limits_{0<\tau<t}\fint\limits_{B_{j+1}\times\{\tau\}}(u+\varepsilon)\,\, dx\Big)^{1-m^-(\beta-1)}.
\end{aligned}
\end{equation}
Combining \eqref{eq5.4}--\eqref{eq5.6} and applying Young's inequality with $\delta \in (0,1)$ and exponents  $\frac{p}{p-(1-\lambda)}$, $\frac{p}{1-\lambda}$, we derive 
\begin{align*}
&\frac{\gamma\,2^j}{r}\int\limits_0^{t}\fint\limits_{\bar B_{j}}u^{(m-1)(p-1)}|D u|^{p-1}\,dx d\tau
\leqslant \gamma 2^{j \gamma} \Big(\frac{t}{r^{p}}\Big)^{\frac{1}{p}}\Big(\sup\limits_{0<\tau<t}\fint\limits_{B_{j+1}\times\{\tau\}}(u+\varepsilon) \,\,dx\Big)^{\frac{(p-1)(m+1)}{p}}\\
&\leqslant
\delta \sup\limits_{0<\tau<t}\fint\limits_{B_{j+1}\times\{\tau\}}(u+\varepsilon)\,\, dx+\frac{\gamma 2^{j\gamma}}{\delta^{\gamma}}\Big(\frac{t}{r^{p}}\Big)^{\frac{1}{1-\lambda}}\leqslant \delta \sup\limits_{0<\tau<t}\fint\limits_{B_{j+1}\times\{\tau\}}(u+\varepsilon)\,\, dx+\frac{\gamma 2^{j\gamma}}{\delta^{\gamma} } \varepsilon,
\end{align*}
where $\gamma$ also depends on $\alpha$ and $\beta$, which are chosen as $\alpha=\min\Big(\frac{1}{2}, \frac{1}{2(p-1)}\Big)$ and $\beta=1+\frac{1}{2m^{-}} \min\big(1, \frac{1-\lambda}{p-1}\big)$.
Combining this estimate with \eqref{recurs_1}, we obtain
\begin{equation*}
J_{j} \leqslant \delta J_{j+1} + \gamma \delta^{-\gamma} 2^{j \gamma}\big(\inf\limits_{0< \tau <t} \fint\limits_{B_{2r}(y)\times\{\tau\}} u\ dx
+ \varepsilon\big), \quad j=0, 1, 2,\ldots
\end{equation*}
Iterating this inequality and choosing $\delta$ sufficiently small, we arrive at \eqref{eq3.1}.
\end{proof}

\noindent The following result provides a variant of the $L^{1}_{\mathrm{loc}}-L^{1}_{\mathrm{loc}}$ Harnack inequality in the case $\lambda>1$ and $p<2$. To the best of our knowledge, this result has not been established previously, even for equation~\eqref{dnleq3}.
%and its proof constitutes one of the novel contributions of this paper.

\begin{theorem}\label{lem3.2}
Let $u$ be a nonnegative, local weak solution to \eqref{eq1.1}-\eqref{eq1.2}, and assume that $\lambda>1$ and $p<2$. Consider a cylinder $Q_{2r,t-s}(y,t)\subset\Omega_T$ and suppose that $u\leqslant\mu^{+}$ in $Q_{2r,t-s}(y,t)$.
Then there exists a constant $\gamma>0$, depending only on the data, such that
\begin{multline}\label{eq3.3}
\sup\limits_{s\leqslant \tau\leqslant t}\fint\limits_{B_{r}(y)}(\mu^+-u(x, \tau))\, dx \leqslant \gamma
\inf\limits_{s\leqslant\tau \leqslant t}\fint\limits_{B_{2 r}(y)} (\mu^+-u(x, \tau))\,dx +\gamma \,
[\mu^+]^\frac{(m-1)(p-1)}{2-p}\Big(\frac{t-s}{r^{p}}\Big)^{\frac{1}{2-p}}.
\end{multline}
\end{theorem}
\begin{proof}
Assume, without loss of generality, that $s=0$. 
Fix $\sigma \in (0,1)$ and let $r \leqslant \rho < \rho(1+\sigma) \leqslant 2r$.
Let $\zeta \in C_0^1(B_{\rho(1+\sigma)}(y))$ be a cutoff function such that $\zeta \equiv 1$ in $B_\rho(y)$, $0 \leqslant \zeta \leqslant 1$, $|D\zeta| \leqslant \frac{1}{\sigma\rho}$. 
We test the integral inequality \eqref{eq1.3} with the function $\zeta$, where the first term can be rewritten as follows:
$$
\int\limits_{B_{\rho(1+\sigma)}(y)} u(x,t)\,\zeta(x)\,dx \Big|_{t_1}^{t_2}=
\int\limits_{B_{\rho(1+\sigma)}(y)}
\big[(\mu^+ - u(x,t_1))-(\mu^+ - u(x,t_2))\big]
\,\zeta(x)\,dx. 
$$
Observe that since $\lambda>1$ and $p<2$, we have $m>1$, and hence $m^-=1$. Therefore, from \eqref{eq1.3}, using the previous equality and the structural condition \eqref{eq1.2}, we obtain for any $0<t_1<t_2<t$
\begin{equation*}
\fint\limits_{B_{\rho}(y)}(\mu^+ - u(x,t_1))\,dx
\leqslant
\fint\limits_{B_{\rho(1+\sigma)}(y)}(\mu^+ - u(x,t_2))\zeta\,dx
+
\frac{\gamma}{\sigma \rho}
\int\limits_{0}^{t}\fint\limits_{B_{\rho(1+\sigma)}(y)}
u^{(m-1)(p-1)}
|D u|^{p-1}\,dx\,d\tau .
\end{equation*}
Choose $t_1$ and $t_2$ such that
$$
\fint\limits_{B_{\rho}(y)}(\mu^+ - u(x,t_1))\,dx=\sup\limits_{0<\tau<t} \fint\limits_{B_{\rho}(y)} (\mu^+ - u(x,\tau))\, dx,
$$
$$
\fint\limits_{B_{2r}(y)} (\mu^+ - u(x,t_2))\,dx= \inf\limits_{0<\tau<t} \fint\limits_{B_{2r}(y)\times\{\tau\}} (\mu^+ - u)\,\,\,dx.
$$
Denote $\mathcal{J}_{\rho} := \sup\limits_{0<\tau<t} \fint\limits_{B_{\rho}(y)} (\mu^+ - u(x,\tau))\, dx$, and we then arrive at
\begin{equation}\label{pre-recurs_ineq_deg}
\mathcal{J}_{\rho}
\leqslant
\gamma\inf\limits_{0<\tau<t} \fint\limits_{B_{2r}(y)\times\{\tau\}} (\mu^+ - u)\,\,\,dx
+
\frac{\gamma}{\sigma \rho}
\int\limits_{0}^{t}\fint\limits_{B_{\rho(1+\sigma)}(y)}
u^{(m-1)(p-1)}
|D u|^{p-1}\,dx\,d\tau .
\end{equation}
To estimate the second term of the previous inequality, we set
$$\varepsilon=\Big([\mu^+]^{(m-1)(p-1)}\frac{t}{\rho^{p}}\Big)^{\frac{1}{2-p}},$$
and choose the parameters $\alpha$ and $\beta$ so that $0< \alpha < \frac{1}{p-1}$, $1<\beta < \min\big(2, \frac{1}{p-1}\big)$.  By the H\"{o}lder's inequality, we obtain
\begin{align}\label{eq8.18}
&\frac{\gamma}{\sigma \rho}\int\limits_0^{t}\fint\limits_{B_{\rho(1+\sigma)}(y)}u^{(m-1)(p-1)}|D u|^{p-1}\,dx d\tau\leqslant
 \frac{\gamma}{\sigma \rho}\biggr(\int\limits_0^{t}\fint\limits_{B_{\rho (1+\sigma)}(y)}\Big(\frac{\tau}{t}\Big)^{\alpha}\frac{u^{(m-1)(p-1)}}
{(\mu^+-u+\varepsilon)^{\beta}}|D u|^{p}\,\zeta^{p}dx d\tau\biggr)^{\frac{p-1}{p}}\notag
\\&\times\biggr(\int\limits_0^{t}\fint\limits_{B_{\rho (1+\sigma)}(y)}\Big(\frac{t}{\tau}\Big)^
{\alpha(p-1)} u^{(m-1)(p-1)}(\mu^+-u+\varepsilon)^{\beta(p-1)}\,dx\,d\tau\biggr)^{\frac{1}{p}}=\gamma \frac{1}{\sigma \rho} A_1^{\frac{p-1}{p}} A_2^{\frac{1}{p}}.
\end{align}
Next, we estimate the right-hand side of \eqref{eq8.18}.
Since $\alpha(p-1)<1$ and $\beta(p-1)<1$, H\"{o}lder's inequality yields
\begin{equation}\label{eq8.20}
\begin{aligned}
A_2&\leqslant [\mu^+]^{(m-1)(p-1)}\int\limits_0^{t}\fint\limits_{B_{\rho(1+\sigma)}(y)}\Big(\frac{t}{\tau}\Big)^{\alpha(p-1)} (\mu^+-u+\varepsilon)^{\beta(p-1)} dx d\tau
\\&\leqslant [\mu^+]^{(m-1)(p-1)} \int\limits_0^{t}\Big(\frac{t}{\tau}\Big)^{\alpha(p-1)}d\tau \sup\limits_{0<\tau<t}\fint\limits_{B_{\rho(1+\sigma)}(y)\times\{\tau\}}(\mu^+-u+\varepsilon)^{\beta(p-1)} dx
\\&\leqslant \gamma[\mu^+]^{(m-1)(p-1)} \,t\, \Big(\sup\limits_{0<\tau<t}\fint\limits_{B_{\rho(1+\sigma)}(y)\times\{\tau\}}(\mu^+-u+\varepsilon) \,\,dx\Big)^{\beta(p-1)}
\\&\leqslant \gamma\rho^{p}\,\varepsilon^{2-p}\, \big(\mathcal{J}_{\rho(1+\sigma)}+\varepsilon\big)^{\beta(p-1)}.
\end{aligned}
\end{equation}
To estimate $A_1$ we apply Lemma~\ref{lem2.4}. Using our choice of $\varepsilon$, the elementary inequality
$(\mu^+-u+\varepsilon)^{\,p-\beta-1} \leqslant \varepsilon^{\,p-\beta-1},\,\,  p-\beta-1<0,$
and since $1<\beta<2$, applying H\"older's inequality, we obtain
\begin{align}\label{eq8.22}
A_1&\leqslant \negthickspace\gamma \negthickspace\fint\limits_{B_{\rho(1+\sigma)}(y)\times\{t\}}\int\limits_u^{\mu^+}(\mu^+-z+\varepsilon)^{1-\beta}\,dz dx+\frac{\gamma}{(\sigma\rho)^{p}}[\mu^+]^{(m-1)(p-1)} \int\limits_0^t\negthickspace\fint\limits_{B_{\rho(1+\sigma)}(y)}\negthickspace\negthickspace\negthickspace(\mu^+-u+\varepsilon)^{p-\beta}\,dxd\tau \notag
\\&\leqslant\gamma \sup\limits_{0<\tau<t}\fint\limits_{B_{\rho(1+\sigma)}(y)\times\{\tau\}}(\mu^+-u+\varepsilon)^{2-\beta}dx+\frac{\gamma}{\sigma^{p}}\frac{1}{\varepsilon^{\beta-1}} \sup\limits_{0<\tau<t}\fint\limits_{B_{\rho(1+\sigma)}(y)\times\{\tau\}}(\mu^+-u+\varepsilon)\,\, dx \notag
\\&\leqslant \gamma \big(\mathcal{J}_{\rho(1+\sigma)}+\varepsilon \big)^{2-\beta}+\frac{\gamma}{\sigma^{p}}\frac{1}{\varepsilon^{\beta-1}} \big(\mathcal{J}_{\rho(1+\sigma)}+\varepsilon\big)
\leqslant \frac{\gamma}{\sigma^{p}}\frac{1}{\varepsilon^{\beta-1}}\big(\mathcal{J}_{\rho(1+\sigma)}+\varepsilon\big).
\end{align}
Combining \eqref{eq8.18}--\eqref{eq8.22} and applying Young's inequality with $\delta \in (0,1)$ and exponents 
$\frac{p}{(p-1)(\beta+1)}$ and $\frac{p}{1-(p-1)\beta}$, we  obtain
\begin{equation*}
\begin{aligned}
\frac{\gamma}{\sigma \rho}\int\limits_0^{t}\fint\limits_{B_{\rho}(y)}u^{(m-1)(p-1)}|D u|^{p-1}\,dx d\tau
&\leqslant\frac{\gamma}{\sigma^{p}}\Big(\frac{1}{\varepsilon^{\beta-1}} \big(\mathcal{J}_{\rho(1+\sigma)}+\varepsilon\big)\Big)^{\frac{p-1}{p}}\Big(\varepsilon^{2-p}\, \big(\mathcal{J}_{\rho(1+\sigma)}+\varepsilon\big)^{\beta(p-1)}\Big)^{\frac{1}{p}}
\\&\leqslant\delta\mathcal{J}_{\rho(1+\sigma)}+\frac{\gamma}{\delta^{\gamma} \sigma^{\gamma}}\varepsilon,
\end{aligned}
\end{equation*}
where $\gamma$ also depends on $\alpha$ and $\beta$, which are chosen as  $\alpha=\frac{1}{2}$ and $\beta=\frac{1}{2}+\min\big(1, \frac{1}{2(p-1)}\big)$.\\
Mixing last estimate and \eqref{pre-recurs_ineq_deg}, we get
\begin{equation*}
\mathcal{J}_{\rho}
\leqslant\delta\mathcal{J}_{\rho(1+\sigma)}+
\frac{\gamma}{\delta^{\gamma} \sigma^{\gamma}}\Big(\inf\limits_{0<\tau<t} \fint\limits_{B_{2r}(y)\times\{\tau\}} (\mu^+ - u)\,\,\,dx
+\Big([\mu^+]^{(m-1)(p-1)}\frac{t}{r^{p}}\Big)^{\frac{1}{2-p}}\Big).
\end{equation*}
From this, by iteration, we obtain the required estimate \eqref{eq3.3}, which completes the proof of Theorem \ref{lem3.2}.
\end{proof}

%\begin{lemma}\label{lem8.2}
%Let $u$ be a non-negative local weak sub-solution to \eqref{eq1.1}, \eqref{eq1.2}, $\lambda> 1$ and $p<2$, then for any $\sigma$, $\delta \in(0, 1)$ and any $r\leqslant\rho< \rho (1+\sigma)\leqslant 2\,r$ there holds
%\begin{equation}\label{eq8.17}
%\frac{1}{\sigma \rho}\int\limits_0^{t}\fint\limits_{B_{\rho}(y)}u^{(m-1)(p-1)}|D u|^{p-1}\,dx d\tau\leqslant \delta \mathcal{J}+\frac{\gamma}{\delta^{\gamma} \sigma^{\gamma}}\Big([\mu^+]^{(m-1)(p-1)}\frac{t}{r^{p}}\Big)^{\frac{1}{2-p}},
%\end{equation}
%where $\mathcal{J}:=\sup\limits_{s<\tau<t}\fint\limits_{B_{\rho(1+\sigma)}(y)} (\mu^+-u(x, \tau))\,dx$ and $\gamma=\gamma(\text{data})>0$.
%\end{lemma}

%%%%%%%%%%%%%%%%%%%%%%%%%%%%%%%%%%%%%%%%%%%%%%%%%%%%%%%%%%%%%%%%%%%%%%%%%%%%%%%%%%%%%%%%%%%%%%%%%%%%%%%%%%%%%%%%%%%%%%%%%%
\section{Reduction of the Oscillation} \label{sec-red-osc}

Fix a point $(x_0, t_0) \in \Omega_T$ and let $b>0$ be a constant to be determined later. Let $\mu^+$, $\mu^-$, and $\omega$ be numbers satisfying the conditions
$$\mu^+\geqslant \sup\limits_{Q_{r, b\,\theta(\omega, r)}}u,\quad \mu^-\leqslant \inf\limits_{Q_{r, b\,\theta(\omega, r)}} u,\quad \mu^+-\mu^-\leqslant \omega,$$
where 
\[Q_{r, b\,\theta(\omega, r)}:=Q_{r, b\,\theta(\omega, r)}(x_0, t_0).\]
As it is well known, the proof of the local H\"older continuity can be presented once the reduction of the oscillation is settled. This is precisely the content of the following proposition - the main step in the proof of Theorem \ref{th1.1}.

\begin{propo}\label{pr1.1}
Let $u$ be a nonnegative local weak solution to \eqref{eq1.1}-\eqref{eq1.2}, and suppose that one of the condition \eqref{eq1.7}, \eqref{eq1.5}, or \eqref{eq1.6} is fulfilled.  Then there exist constants $b>0$, and $\eta_0, \epsilon_0, \sigma_0\in (0, 1)$ depending only on the data, such that if
\begin{equation}\label{eq1.8}
\mu^+\leqslant (1+\eta_0)\,\omega,
\end{equation}
then
\begin{equation}\label{eq1.9}
\osc\limits_{Q_{\epsilon_0 r,  b\,\theta(\sigma_0\omega, \epsilon_0 r)}} u \leqslant \sigma_0\,\omega.
\end{equation}
\end{propo}

\noindent The proof of Proposition~\ref{pr1.1} is presented in Subsections~\ref{sec-case1}, \ref{sec-case2}, and~\ref{sec-case3}, corresponding to the three cases  \eqref{eq1.7}, \eqref{eq1.5}, and \eqref{eq1.6}, respectively. We first introduce some common notation and preliminary results that will be used throughout the proof.

\noindent For constants $b>0$ and $\epsilon\in(0,1)$ to be specified later, we define the
cylinder
\begin{equation*}
Q^{\omega}_{\epsilon r}:=Q_{\epsilon r, b\,\theta(\omega,\epsilon r)}(x_0, t_0)\subset Q_{r,b\,\theta(\omega, r)}. 
\end{equation*}
We note that, without loss of generality, we may assume that
\begin{equation}\label{eq4.1}
\mu^{+} \geqslant \frac{1}{2}\,\omega,
\end{equation}
since otherwise the estimate \eqref{eq1.9} follows immediately. Indeed,
\begin{equation*}
\osc\limits_{Q_{\epsilon_0 r,  b\,\theta(\sigma_0\omega, \epsilon_0 r)}} u \leqslant \osc\limits_{Q_{r,  b\,\theta(\omega,r)}} u \leqslant\mu^+\leqslant\sigma_0\omega
\end{equation*}
with $\sigma_0=\tfrac{1}{2}$ and some $\epsilon_0\in(0,1)$.

The remainder of this section is then devoted to the proof of Proposition~\ref{pr1.1}.

\subsection{Case \eqref{eq1.7}: $\lambda<1$ and $p+N(\lambda-1)>0$ }\label{sec-case1}

\subsubsection{Two Alternatives}

Let $b\in(0,1)$ be a constant to be specified later, depending only on the data. We assume that one of the following alternatives holds for some $\nu\in(0,1)$, depending only on the data and on $b$. 
Either
\begin{equation}\label{eq4.4}
\bigl|Q^{\omega}_{\frac{1}{2}\epsilon r}\cap\{u\geqslant \mu^{+}-\tfrac{1}{4}\omega\}\bigr|\leqslant\nu\,\bigl|Q^{\omega}_{\frac{1}{2}\epsilon r}\bigr|,
\end{equation}
or
\begin{equation}\label{eq4.5}
\bigl|Q^{\omega}_{\frac{1}{2}\epsilon r}\cap\{u\geqslant \mu^{+}-\tfrac{1}{4}\omega\}\bigr|\geqslant\nu\,\bigl|Q^{\omega}_{\frac{1}{2}\epsilon r}\bigr|.
\end{equation}

\subsubsection{Analysis of the First Alternative}
Choose $\nu=\frac{1}{\gamma}\,b^{\frac{N}{p}}$. Then, by  De Giorgi type Lemma~\ref{lem2.6}, estimate \eqref{eq4.4} yields
$$u\leqslant \mu^+-\frac{1}{8}\,\omega,\quad \text{in}\quad Q^{\omega}_{\frac{1}{4} \epsilon r}.$$
Consequently,
$$u-\inf\limits_{Q^{\omega}_{\frac{1}{4}\epsilon r }}u\leqslant u-\mu^-\leqslant \mu^+-\mu^--\frac{1}{8}\,\omega\leqslant\frac{7}{8}\,\omega,\quad \text{in}\quad Q^{\omega}_{\frac{1}{4} \epsilon r}.$$
Therefore,
\begin{equation}\label{eq4.6}
\osc\limits_{Q^{\omega}_{\frac{1}{4}\epsilon r }} u \leqslant \frac{7}{8}\,\omega,
\end{equation}
which proves Proposition \ref{pr1.1}.

\subsubsection{Analysis of the Second Alternative}

Inequality \eqref{eq4.5} yields
\begin{equation}\label{eq4.5*}
\bigl|B_{\frac{1}{2}\epsilon r}(x_0)\cap\{u(\cdot,\bar{t})\geqslant \mu^{+}-\tfrac{1}{4}\omega\}\bigr|\geqslant\nu\,\bigl|B_{\frac{1}{2}\epsilon r}(x_0)\bigr|,
\qquad\nu=\frac{1}{\gamma}\,b^{\frac{N}{p}}.
\end{equation}
for a time level $\bar{t}\in (t_0- \,b\,(\frac{1}{2}\epsilon \,r)^{p}\,\omega^{1-\lambda}, t_0)$.

\begin{lemma}[\bf Propagation of Positivity in Measure]\label{lem4.1}
Let inequality \eqref{eq4.5*} hold. Then there exist constants $\bar{\epsilon}$, $\alpha_0\in (0, 1)$ depending only on the data and on $b$, such that for all $t\in(t_0-b \,(\epsilon r)^{p}\,\omega^{1-\lambda},t_0)$ there holds
\begin{equation}\label{eq4.7}
|B_{\epsilon r}(x_0)\cap\{u(\cdot, t)\geqslant \bar{\epsilon}\omega\}|\geqslant \alpha_0\,|B_{\epsilon r}(x_0)|.
\end{equation}
\end{lemma}
\begin{proof}
First, we note that by \eqref{eq1.8} and \eqref{eq4.1}, $\frac{1}{2}\omega\leqslant \mu^+\leqslant 2 \omega$. By inequality \eqref{eq4.5*}, we obtain 
$$\fint\limits_{B_{\frac{1}{2} \epsilon r}(x_0)} u(x, \bar{t})\,dx\geqslant
\frac{1}{|B_{\frac{1}{2} \epsilon r}(x_0)|}\int\limits_{B_{\frac{1}{2}\epsilon r}(x_0)\cap\{u(\cdot, \bar{t})\geqslant \mu^+-\frac{1}{4}\,\omega\}} u(x, \bar{t})\,dx\geqslant
\nu \biggr(\mu^+-\frac{1}{4}\omega\biggr)\geqslant \frac{1}{\gamma}\,b^{\frac{N}{p}}\omega.$$
We now apply the $L^{1}_{\mathrm{loc}}-L^{1}_{\mathrm{loc}}$ Harnack inequality, Theorem~\ref{lem3.1}, in the cylinder $Q^{\omega}_{\epsilon r}$:
\begin{align}\label{eq4.8}
\frac{1}{\gamma}\,b^{\frac{N}{p}}\omega& \leqslant 
\fint\limits_{B_{\frac{1}{2} \epsilon r}(x_0)} u(x, \bar{t})\,dx 
\leqslant \gamma\inf\limits_{t_0-b\, (\epsilon r)^{p} \omega^{1-\lambda}\leqslant t\leqslant t_0}\fint\limits_{B_{\epsilon r}(x_0)} u(x, t)\,dx +\gamma \omega\, b^{\frac{1}{1-\lambda}} 
\notag\\ &=\frac{\gamma}{|B_{\epsilon r}(x_0)|}\inf\limits_{t_0-b\, (\epsilon r)^{p} \omega^{1-\lambda}\leqslant t\leqslant t_0}\int\limits_{B_{\epsilon r}(x_0)\cap\{u(\cdot, t)< \bar{\epsilon}\omega\}} u(x, t)\,dx 
\notag\\&+ \frac{\gamma}{|B_{\epsilon r}(x_0)|}\inf\limits_{t_0-b\, (\epsilon r)^{p} \omega^{1-\lambda}\leqslant t\leqslant t_0}\int\limits_{B_{\epsilon r}(x_0)\cap\{u(\cdot, t)\geqslant \bar{\epsilon}\omega\}} u(x, t)\,dx +\gamma \omega\, b^{\frac{1}{1-\lambda}}
\notag\\&\leqslant\gamma\bar{\epsilon}\omega + \gamma 2 \omega \frac{|B_{\epsilon r}(x_0)\cap\{u(\cdot, t)\geqslant \bar{\epsilon}\,\omega\}|}{|B_{\epsilon r}(x_0)|} +\gamma \omega\, b^{\frac{1}{1-\lambda}}\notag
\\&\leqslant \gamma \bar{\epsilon} \omega+ \gamma b^{\frac{1}{1-\lambda}}\,\omega+
\gamma\,\omega\frac{|B_{\epsilon r}(x_0)\cap\{u(\cdot, t)\geqslant \bar{\epsilon}\,\omega\}|}{|B_{\epsilon r}(x_0)|},
\end{align}
for all $t\in (t_0-b\,(\epsilon r)^{p}\,\omega^{1-\lambda}, t_0)$.
Using the assumption $p+N(\lambda-1)>0$, we choose the constants $b$ and
$\bar{\epsilon}$ so that
\begin{equation*}
\gamma \,b^{\frac{1}{1-\lambda}-\frac{N}{p}}\leqslant \frac{1}{2\gamma},\quad \text{and}\quad \gamma\,\bar{\epsilon}=\frac{1}{4\gamma}\,b^{\frac{N}{p}}.
\end{equation*}
With this choice, we obtain \eqref{eq4.7} with $\alpha_0=\frac{1}{4\gamma}\,b^{\frac{N}{p}}$, which completes the proof of the lemma.
\end{proof}

Continuing the proof of Proposition~\ref{pr1.1}, we apply Lemma~\ref{lem3.3}. Since \eqref{eq4.7} holds, Lemma~\ref{lem3.3} yields the existence of constants $\varepsilon, \delta, \sigma\in(0,1)$, depending only on the data and on $b$, such that the following estimate holds:
$$u(x, t)\geqslant \sigma\bar{\epsilon}\,\omega,\quad x\in B_{\frac{1}{2} \epsilon r}(x_0),$$
for all $t$ in the interval
\begin{equation*}
t_0-\,b\,(\epsilon r)^{p}\,\omega^{1-\lambda}+\frac{1}{2}(1-\varepsilon)\delta(\epsilon r)^{p}\,(\bar{\epsilon}\omega)^{1-\lambda} = t_0-\big(b-\tfrac{1}{2}(1-\varepsilon)\delta\,\bar{\epsilon}^{1-\lambda}\big)(\epsilon r)^{p}\omega^{1-\lambda}\leqslant  t\leqslant t_0,
\end{equation*}
Choosing $b=(1-\varepsilon)\delta\,\bar{\epsilon}^{1-\lambda}$, we obtain
\begin{equation*}
t_0-\frac{1}{2}\,b\,(\epsilon r)^{p}\,\omega^{1-\lambda}\leqslant t_0-b\,\big(\tfrac{1}{2}\epsilon r\big)^{p}\,\omega^{1-\lambda}\leqslant t\leqslant t_0.
\end{equation*}
Finally, using \eqref{eq1.8}, we get
\begin{equation}\label{eq4.9}
\osc\limits_{Q^{\omega}_{\frac{1}{2}\epsilon r}} u\leqslant (1+\eta_0-\sigma\bar{\epsilon})\omega=\big(1-\tfrac{1}{2}\sigma\bar{\epsilon}\big)\omega,
\end{equation}
provided that $\eta_0= \frac{1}{2}\sigma\bar{\epsilon}$. 

\subsubsection{Pasting of Two Alternatives Together}
Combining \eqref{eq4.6} and \eqref{eq4.9} and choosing 
$$\sigma_0=\max\left(1-\frac{\sigma\bar{\epsilon}}{2} ,\,\frac{7}{8}\right),\quad 
\epsilon_0= \frac{\epsilon}{4},$$
we obtain
\begin{equation}\label{eq4.10}
 \osc\limits_{Q^{\sigma_0 \omega}_{\epsilon_0 r}} u  \leqslant \sigma_0\,\omega.
\end{equation}
%\textcolor{pink}{\begin{equation*}
%$Q^{\omega}_{\epsilon r}:=Q^-_{\epsilon r, b\,(\epsilon %r)^{p}\omega^{1-\lambda}}(x_0, t_0)\subset Q^-_{r, \theta(\omega, r)}. 
%\end{equation*}
%$$Q^-_{\epsilon r, b\,(\epsilon r)^{p}\omega^{1-\lambda}}(x_0, t_0)=B_{\epsilon r}(x_0)\times (t_0-b\,(\epsilon r)^{p}\omega^{1-\lambda}, t_0).$$
%$Q^-_{\epsilon_0 r, \theta(\sigma_0 \omega, \epsilon_0 r)}=Q^{\sigma_0 \omega}_{\epsilon_0 r}$??}
This proves Proposition \ref{pr1.1} in the case $\lambda<1$ and $p+N(\lambda-1)>0$.

\subsection{Case \eqref{eq1.5}: $\lambda>1$ and $\frac{2N}{N+1}<p<2$}\label{sec-case2}

\subsubsection{Two Alternatives}

Let $b\in(0,1)$ be a constant to be specified later, depending only on the data. We assume that one of the following alternatives holds for some $\nu\in(0,1)$, depending only on the data and on $b$. 
Either
\begin{equation}\label{eq6.3}
\big| Q^{\omega}_{\frac{1}{2}\epsilon r}\cap\{u\leqslant \tfrac{1}{4}\,\omega\}\big|\leqslant \nu \big|Q^{\omega}_{\frac{1}{2}\epsilon r}\big|,
\end{equation}
or
\begin{equation}\label{eq6.4}
\big| Q^{\omega}_{\frac{1}{2}\epsilon r}\cap\{u\leqslant \tfrac{1}{4}\,\omega \}\big|\geqslant \nu \big|Q^{\omega}_{\frac{1}{2}\epsilon r}\big|.
\end{equation}

\subsubsection{Analysis of the First Alternative}
In the  case \eqref{eq6.3}, choosing $\nu=\frac{1}{\gamma}\,b^{\frac{N}{p}}$ and applying De Giorgi type Lemma \ref{lem2.5}, we obtain 
\begin{equation*}
u\geqslant \frac{1}{8}\,\omega,\quad \text{in}\quad  Q^{\omega}_{\frac{1}{4}\epsilon r}.
\end{equation*}
Consequently, using \eqref{eq1.8} and choosing $\eta_0=\frac{1}{16}$, we derive
\begin{equation}\label{eq6.5}
\osc\limits_{Q^{\omega}_{\frac{1}{4}\epsilon r}} u\leqslant \biggr(\eta_0+\frac{7}{8}\biggr)\,\omega=\frac{15}{16}\,\omega.
\end{equation}

\subsubsection{Analysis of the Second Alternative}

he second alternative \eqref{eq6.4} yields
\begin{equation}\label{eq6.4*}
\big| B_{\frac{1}{2}\epsilon r}(x_0)\cap\{u(\cdot, \bar{t})\leqslant \tfrac{1}{4}\,\omega\}\big|\geqslant \nu \big|B_{\frac{1}{2}\epsilon r}(x_0)\big|,\qquad\nu=\frac{1}{\gamma}\,b^{\frac{N}{p}},
\end{equation}
for some time level $\bar{t}\in(t_0-b\,(\frac{1}{2}\epsilon r)^{p}\,\omega^{1-\lambda}, t_0)$. 

\begin{lemma}[\bf Propagation of Positivity in Measure]\label{lem6.1}
Let inequality \eqref{eq6.4*} hold. Then there exist constants $\bar{\epsilon}$, $\alpha_0\in (0, 1)$ depending only on the data and on $b$, such that for all $t\in(t_0-b \,(\epsilon r)^{p}\,\omega^{1-\lambda},t_0)$ there holds
%\begin{equation}\label{eq6.6}
%|B_{\epsilon r}(x_0)\cap\big\{u(\cdot, t)\geqslant \mu^+-\bar{\epsilon}\,\omega\big\}|\leqslant (1-\alpha_0)|B_{\epsilon r}(x_0)|.
%\end{equation}
\begin{equation}\label{eq6.6}
\big|B_{\epsilon r}(x_0)\cap\big\{u(\cdot, t)\leqslant \mu^+-\bar{\epsilon}\,\omega\big\}\big|\geqslant \alpha_0\big|B_{\epsilon r}(x_0)\big|.
\end{equation}
\end{lemma}
\begin{proof}
First, we note that by \eqref{eq1.8} and \eqref{eq4.1}, $\frac{1}{2}\omega\leqslant \mu^+\leqslant 2 \omega$. We now apply Theorem~\ref{lem3.2} in the cylinder $Q^{\omega}_{\epsilon r}$ and obtain \eqref{eq3.3} in this cylinder. From this, for any $t\in(t_0-b\,(\epsilon r)^{p}\,\omega^{1-\lambda}, t_0)$ the following estimate holds:
\begin{equation*}
\begin{aligned}
\frac{1}{\gamma}\,b^{\frac{N}{p}}\,\omega&\leqslant \nu\,(\mu^+-\frac{1}{4}\,\omega)\leqslant \fint\limits_{B_{\frac{1}{2}\epsilon r}(x_0)}(\mu^+-u(x, \bar{t}))\,dx
\\&\leqslant\gamma \inf\limits_{t_0-b\,(\epsilon r)^{p} \omega^{1-\lambda}\leqslant t \leqslant t_0}\fint\limits_{B_{\epsilon r}(x_0)}(\mu^+-u(x, t))\,dx
+\gamma\,\omega b^{\frac{1}{2-p}}
\\&\leqslant \gamma \omega \Big(\bar{\epsilon}+b^{\frac{1}{2-p}}\Big)+\omega\frac{|B_{\epsilon r}\cap\{\mu^+-u(\cdot, t)\geqslant \bar{\epsilon}\,\omega\}|}{|B_{\epsilon r}(x_0)|}.
\end{aligned}
\end{equation*}
Since $p>\frac{2N}{N+1}$, we choose the constants $b$ and $\bar{\epsilon}$ from the conditions
$$\gamma b^{\frac{1}{2-p}-\frac{N}{p}}\leqslant \frac{1}{2\gamma},\quad  \gamma \,\bar{\epsilon}\leqslant \frac{1}{4\gamma}\,b ^{\frac{N}{p}}.$$
With this choice, we obtain \eqref{eq6.6}, which completes the proof of the lemma.
\end{proof}

To complete the proof of Proposition~\ref{pr1.1}, we apply
Lemma~\ref{new_exp_lem}. Recall that, by \eqref{eq1.8} and \eqref{eq4.1}, we have $\frac{1}{2}\omega\leqslant \mu^+\leqslant 2 \omega$. Hence, by \eqref{eq6.6}, Lemma~\ref{new_exp_lem} yields the existence of constants $\sigma, \delta \in (0, 1)$, such that
\begin{equation*}
u(x, t)\leqslant \mu^+-\sigma\,\bar{\epsilon}\,\omega,\quad x\in B_{\frac{1}{2}\epsilon r}(x_0),
\end{equation*}
for all $t$ in the interval
\begin{equation*}
t_0-b\,(\epsilon r)^{p}\omega^{1-\lambda}+\delta\,\bar{\epsilon}^{2-p}(\epsilon r)^{p}\omega^{1-\lambda} = t_0-(b-\delta\,\bar{\epsilon}^{2-p})(\epsilon r)^{p}\omega^{1-\lambda}\leqslant  t\leqslant t_0.
\end{equation*}
Choosing $\bar{\epsilon}$ sufficiently small so that $\delta\,\bar{\epsilon}^{2-p}<\frac{1}{2}$, and setting $b=2\delta\,\bar{\epsilon}^{2-p}$, we deduce:
\begin{equation}\label{eq6.12}
u(x, t)\leqslant \mu^+-\sigma\,\bar{\epsilon}\,\omega,\quad x\in B_{\frac{1}{2}\epsilon r}(x_0),
\end{equation}
for all $t$ satisfying
\begin{equation*}
t_0-\frac{1}{2}\,b\,(\epsilon r)^{p}\,\omega^{1-\lambda}\leqslant t_0-b\,\big(\tfrac{1}{2}\epsilon r\big)^{p}\,\omega^{1-\lambda}\leqslant t\leqslant t_0.
\end{equation*}
Finally, from \eqref{eq6.12}, we obtain
\begin{equation}\label{eq6.14}
\osc\limits_{Q_{\frac{1}{2}\epsilon r}^\omega} u 
\leqslant (1-\sigma\,\bar{\epsilon})\omega.
\end{equation}
%\textcolor{magenta}{$$u-\inf\limits_{Q_{\frac{1}{2}\epsilon r,\frac{1}{2}\,b\,(\epsilon r)^{p}\,\omega^{1-\lambda}}}u\leqslant u-\mu^- \leqslant \mu^+-\mu^- -\eta\,\omega \leqslant (1-\eta)\omega\quad\text{in }Q^{\omega}_{\frac{1}{2}\epsilon r}$$}

\subsubsection{Pasting of Two Alternatives Together}
Combining \eqref{eq6.5} and \eqref{eq6.14} and choosing 
$$\sigma_0=\max\biggr(1-\sigma\,\bar{\epsilon},\frac{15}{16}\biggr),\qquad\epsilon_0=\min\biggr(\frac{\epsilon}{2}\,\sigma_0^\frac{\lambda-1}{p},\,\frac{1}{4}\biggr),$$
we arrive at
\begin{equation}\label{eq6.15}
 \osc\limits_{Q^{\sigma_0 \omega}_{\epsilon_0 r}} u\leqslant \sigma_0\,\omega,
\end{equation}
which completes the proof of Proposition \ref{pr1.1}.

\subsection{Case \eqref{eq1.6}: $\lambda>1$ and $p>2$}\label{sec-case3}
In this case, we take \(b>1\) and define the cylinder
\[
\bar{Q}^{\omega}_{\epsilon r} := B_{\epsilon r}(x_0) \times \big(t_0 - b(\epsilon r)^p \omega^{1-\lambda}, \, t_0 - (b-1)(\epsilon r)^p \omega^{1-\lambda} \big) \subset Q^{\omega}_{\epsilon r}.
\]

\subsubsection{Two Alternatives}

The following two alternative cases are possible: either 
\begin{equation}\label{eq44.17}
\big| \bar{Q}^{\omega}_{\epsilon r}\cap\{u\geqslant \mu^+-\tfrac{1}{4}\,\omega\}\big|\leqslant \nu |\bar{Q}^{\omega}_{\epsilon r}|,
\end{equation}
 or
\begin{equation}\label{eq44.18}
\big|\bar{Q}^{\omega}_{ \epsilon r}\cap\{u\geqslant \mu^+-\tfrac{1}{4}\,\omega\}\big|\geqslant \nu |\bar{Q}^{\omega}_{ \epsilon r}|
\end{equation}
with some $\nu \in (0, 1)$ to be defined depending only on the data.

\subsubsection{Analysis of the First Alternative}
In the  case \eqref{eq44.17}, choosing $\nu=\frac{1}{\gamma}$ and using De Giorgi type Lemma \ref{lem2.6} with $b=1$ and $\tau=t_0-(b-1)(\epsilon r)^p \omega^{1-\lambda}$, we obtain 
$$u\leqslant \mu^+-\frac{1}{8}\omega,\quad \text{in}\quad  \bar{Q}^{\omega}_{\frac{1}{2}\epsilon r}.$$
A De Giorgi type lemma involving the initial data, Lemma  \ref{lem2.7} ensure the existence of $\xi\in (0, \frac{1}{8})$ depending only on the data and  $b$  such that
$$u\leqslant \mu^+-\frac{1}{4}\xi\,\omega,\quad \text{in}\quad  Q^{\omega}_{\frac{1}{4}\epsilon r},$$
and consequently
\begin{equation}\label{eq44.19}
\osc\limits_{Q^{\omega}_{\frac{1}{4}\epsilon r}} u\leqslant \Big(1-\frac{1}{4}\xi\Big)\,\omega.
\end{equation}

\subsubsection{Analysis of the Second Alternative} 

By \eqref{eq4.1}, the second alternative \eqref{eq44.18} yields
\begin{equation}\label{eq44.20}
\big|B_{\epsilon r}(x_0)\cap\big\{u(\cdot, \bar{t})\geqslant \tfrac{1}{4}\omega\big\}\big|\geqslant
\big|B_{\epsilon r}(x_0)\cap\{u(\cdot, \bar{t})\geqslant \mu^+-\tfrac{1}{4}\,\omega\}\big|\geqslant \nu |B_{\epsilon r}(x_0)|,
\end{equation}
for the time level
$t_0-b\,(\epsilon r)^{p}\,\omega^{1-\lambda}\leqslant \bar{t}\leqslant t_0-(b-1)(\epsilon r)^{p}\omega^{1-\lambda}$. 
Expansion of positivity Lemma~\ref{deg_exp_lem} guarantees the existence of constants $\sigma\in(0,1)$ and $\delta >1$ depending only on the data and on $b$, such that
\begin{equation}\label{eq44.21}
u(x, t)\geqslant \frac{\sigma}{4}\,\omega,\quad x\in B_{\frac{1}{2}\epsilon r}(x_0)
\end{equation}
for all time levels
$$
t_0-b\,(\epsilon r)^{p}\,\omega^{1-\lambda}+\frac{1}{2}\delta(\epsilon r)^p \big(\tfrac{1}{4}\omega\big)^{1-\lambda}\leqslant t\leqslant t_0-b\,(\epsilon r)^{p}\,\omega^{1-\lambda}+\delta(\epsilon r)^p \big(\tfrac{1}{4}\omega\big)^{1-\lambda}.
$$
Choosing $b=4^{\lambda-1}\delta$, we deduce:
\begin{equation*}
t_0-\frac{1}{2}\,b\,(\epsilon r)^{p}\,\omega^{1-\lambda}\leqslant t_0-b\,\big(\tfrac{1}{2}\epsilon r\big)^{p}\,\omega^{1-\lambda}\leqslant t\leqslant t_0.
\end{equation*}
Using \eqref{eq1.8},  \eqref{eq44.21}, we obtain
\begin{equation}\label{eq44.23}
\osc\limits_{Q^{\omega}_{\frac{1}{2}\epsilon r}} u\leqslant \biggr(1+\eta_0-\frac{\sigma}{4}\biggr)\omega=\biggr(1-\frac{\sigma}{8}\biggr)\omega,
\end{equation}
provided that $\eta_0=\frac{\sigma}{8}$. 

\subsubsection{Pasting of Two Alternatives Together}

Combining \eqref{eq44.19} and \eqref{eq44.23} and choosing 
$$
\sigma_0=\max\biggr(1-\frac{\xi}{4},1-\frac{\sigma}{8}\biggr), \qquad\epsilon_0=\min\biggr(\frac{\epsilon}{2}\,\sigma_0^\frac{\lambda-1}{p},\frac{1}{4}\biggr),
$$
we arrive at
\begin{equation}\label{eq44.15}
\osc\limits_{Q^{\sigma_0\omega}_{\epsilon_0 r}} u \leqslant  \sigma_0\,\omega,
\end{equation}
which completes the proof of Proposition \ref{pr1.1}.
%%%%%%%%%%%%%%%%%%%%%%%%%%%%%%%%%%%%%%%%%%%%%%%%%%%%%%%%%%%%%%%%%%%%%%%%%%%%%%%%%%%%%%%%%%%%%%%%%%%%%%%%%%%%%%%%%%%%%%%%%%
\section{Proof of Theorem~\ref{th1.1}}

\subsection{Preliminary settings}

Fix $(x_0, t_0)\in \Omega_T$ and construct the cylinder 
%\begin{equation*}
%Q_{\rho}:= 
%\begin{cases}
%B_{\rho}(x_0) \times(t_0- M^{1-\lambda}\rho^p, t_0),   &\text{if } \lambda<1, \\
%B_{\rho}(x_0) \times(t_0- \rho, t_0)\subset \Omega_T,   &\text{if } \lambda>1,
%\end{cases} 
%\end{equation*}
\begin{equation*}
Q_{\rho}:= 
B_{\rho}(x_0) \times(t_0- \rho, t_0)\subset \Omega_T.
\end{equation*}
%where $M:=\sup\limits_{\Omega_T}u$. 
Choose $r$ and $R$ such that $0<r<R$ and $Q_{8r} \subset Q_{R}\subset \Omega_T$. If for some constant $A>1$ (to be specified later), for all $0<\rho\leqslant r$, 
\begin{equation*}
\osc\limits_{Q_\rho} u\leqslant A\rho^\beta,\quad \text{ where } \quad \beta:=
\begin{cases}
1,   &\text{if } \lambda<1, \\
\frac{p-1}{\lambda-1},   &\text{if } \lambda>1,
\end{cases} 
\end{equation*}
then the H\"{o}lder continuity follows immediately. Otherwise, we assume that there exists $\rho_0 \in (0, r]$ such that
\begin{equation*}
\osc\limits_{Q_{\rho_0}} u \geqslant A\rho_0^\beta.
\end{equation*}
Set 
\begin{equation*}
\mu^+_{-1}:=\sup\limits_{Q_{\rho_0}} u, \qquad
\mu^-_{-1}:=\inf\limits_{Q_{\rho_0}} u, \qquad
\omega_0:=\mu^+_{-1}.
\end{equation*}
Then
\begin{equation}\label{eq9.5}
A\rho_0^\beta\leqslant\osc\limits_{Q_{\rho_0}} u\leqslant \omega_0.
\end{equation}
For $b>0$ to be defined, construct the intrinsic cylinder
\begin{equation*}
Q_{\rho_0, b\,\theta(\omega_0, \rho_0)}:=Q_{\rho_0, b\,\theta(\omega_0, \rho_0)}(x_0, t_0).
\end{equation*}
Observe that, in the case $\lambda<1$, we have
%$$b\,\omega_0^{1-\lambda}\rho_0^p\leqslant b\,M^{1-\lambda}\,\rho^{p}_0\leqslant M^{1-\lambda}\,\rho^{p}_0.$$
\begin{equation*}
b\,\omega_0^{1-\lambda}\rho_0^p\leqslant b\,M^{1-\lambda}r^{p-1}\,\rho_0\leqslant \rho_0, \qquad M:=\sup_{Q_R}u,
\end{equation*}
provided that $r$ is chosen sufficiently small so that $b\,M^{1-\lambda}r^{p-1}\leqslant1$.

\vskip0.2cm \noindent 
In the case $\lambda>1$, using \eqref{eq9.5}, we obtain
%$$b\,\theta(\omega_0, \rho_0)= b\,\omega_0^{1-\lambda}\rho_0^p \leqslant b\,\rho_0\leqslant \rho_0.$$ 
\begin{equation*}
b\,\omega_0^{1-\lambda}\rho_0^p\leqslant \frac{b}{A^{\lambda-1}}\rho_0\leqslant \rho_0,
\end{equation*}
provided that $A$ is large enough to ensure $bA^{1-\lambda}\leqslant1$.
Consequently,
$$Q_{\rho_0, b\theta(\omega_0, \rho_0)}\subset Q_{\rho_0}$$
in both cases $\lambda>1$ and $\lambda<1$.

\subsection{Main Proposition}

Theorem \ref{th1.1} is a consequence of the following result.

\begin{propo}\label{pr9.1}
There exist constants $b>0$ and $\epsilon,\sigma\in(0,1)$ such that the following
holds.
Construct the sequences
\begin{equation*}
r_j:=\epsilon^{j} \rho_0, \qquad\omega_j:=\sigma^j \omega_0, \qquad\theta_j:=\theta(\omega_j, r_j)=r^{p}_j\,\omega^{1-\lambda}_j, \qquad j=0, 1, 2, ...
\end{equation*}
And define the family of intrinsic cylinders 
\begin{equation*}
Q_j:=Q_{r_{j}, b\,\theta_{j}}(x_0,t_0), \qquad \mu^+_j:=\sup\limits_{Q_j}u, \qquad \mu^-_j:=\inf\limits_{Q_j} u.
\end{equation*}
Then for all $j=0,1,2,\ldots$, the following assertions hold:
\begin{equation}\label{eq9.11}
Q_{j+1}\subset Q_{j},\qquad
\osc\limits_{Q_{j}} u=\mu^+_j-\mu^-_j\leqslant \omega_j.
\end{equation}
\end{propo}

\begin{proof}
We proceed by induction on $j$.

\medskip
\noindent\textbf{Step $j=0$.}
By definition, we have
\begin{equation*}
\begin{aligned}
&r_0=\rho_0,\quad \omega_0=\omega_0,\quad \theta_0=\omega^{1-\lambda}_0\,\rho^{p}_0,\quad Q_0=Q_{\rho_0, b\,\omega^{1-\lambda}_0\,\rho^{p}_0}
\\&\mu^+_0=\sup\limits_{Q_{0}}u,\quad \mu^-_0=\inf\limits_{Q_{0}} u,\quad \osc\limits_{Q_{0}} u=\mu^+_0-\mu^-_0\leqslant \mu^+_{-1}= \omega_0.
\end{aligned}
\end{equation*}
Hence, \eqref{eq9.11} holds for $j=0$.

\medskip
\noindent\textbf{Step $j=1$.}
By the choice of $\omega_0$, we have
\begin{equation}\label{eq9.12}
\mu^+_0\leqslant \mu^+_{-1} = \omega_0 \leqslant (1+\eta_0)\,\omega_0,
\end{equation}
where $\eta_0$ is the constant defined in Proposition \ref{pr1.1}. Applying Proposition~\ref{pr1.1}, we obtain the existence of constants $b>0$ and $\sigma_0,\epsilon_0\in(0,1)$ such that
\begin{equation}\label{eq9.13}
\osc\limits_{Q_{1}} u\leqslant \omega_1,\qquad r_1=\epsilon_0 \rho_0,\,\,\,\omega_1=\sigma_0\,\omega_0.
\end{equation}
Therefore, \eqref{eq9.11} holds for $j=1$.

\medskip
\noindent\textbf{Induction step.}
We may iterate this procedure, assuming at each step that
\begin{equation}\label{eq9.14}
\mu^+_j\leqslant (1+\eta_0)\,\omega_j,
\end{equation}
and thereby obtain \eqref{eq9.11} with the same constants $\epsilon_0$ and $\sigma_0$ as in \eqref{eq9.13}. Assume now that $j_0\geqslant 1$ is the first number for which \eqref{eq9.14} fails, that is,
\begin{equation}\label{eq9.15}
\mu^+_{j_0}\geqslant (1+\eta_0)\,\omega_{j_0}\quad \Rightarrow \quad \mu^-_{j_0}\geqslant \frac{\eta_0}{1+\eta_0}\mu^+_{j_0}.
\end{equation}
Since \eqref{eq9.14} holds for $j=j_0-1$ (and, in particular, for $j=0$ by \eqref{eq9.12}), we obtain
\begin{equation}\label{eq9.16}
\eta_0\,\omega_{j_0}\leqslant \frac{\eta_0}{1+\eta_0}\mu^+_{j_0}\leqslant \mu^-_{j_0}\leqslant \mu^+_{j_0-1}\leqslant (1+\eta_0)\,\omega_{j_0-1}=\frac{1+\eta_0}{\sigma}\omega_{j_0}.
\end{equation}
In this case, we perform the change of variables
\begin{equation}
\tau=\omega_{j_0}^{\lambda-1}t,\quad v(x,\tau)=\frac{u(x,\omega_{j_0}^{1-\lambda}\tau)}{\omega_{j_0}},
\end{equation}
which maps the cylinder $Q_{j_0}$ into the cylinder $Q_{r_{j_0}, b\,r^{p}_{j_0}}(x_0, \tau_0)$, where $\tau_0:=\omega_{j_0}^{\lambda-1}t_0$. The rescaled function $v$ is uniformly bounded:
\begin{equation}\label{eq9.17}
\eta_0\leqslant v(x,\tau)\leqslant\frac{1+\eta_0}{\sigma}\qquad \text{in } Q_{r_{j_0}, b\,r^{p}_{j_0}}(x_0, \tau_0)
\end{equation}
Consequently, the rescaled function $v$ satisfies a $p$-Laplace type equation with
$p>1$ in the cylinder $Q_{r_{j_0},\,b\,r_{j_0}^{p}}(x_0,\tau_0)$, namely,
\begin{equation*}
\begin{cases}
v_\tau - \operatorname{div}\hat{\mathbf{A}}(x, \tau, v, D v)=0,\qquad (x, \tau)\in Q_{r_{j_0}, b\,r^{p}_{j_0}}(x_0, \tau_0),\\
\hat{\mathbf{A}}(x, \tau, v, D v) D v\geqslant \hat{K}_1\,|D v|^{p},\\
|\hat{\mathbf{A}}(x, \tau, v, D v)|\leqslant \hat{K}_2\,|D v|^{p-1},
\end{cases}
\end{equation*}
where the constants $\hat{K}_1$ and $\hat{K}_2$ depend only on the data, $\eta_0$, and
$\sigma$.By \eqref{eq9.17}, we have:
\begin{equation*}
\osc\limits_{Q_{r_{j_0}, b\,r^{p}_{j_0}}(x_0, \tau_0)}v\leqslant \frac{1+\eta_0(1-\sigma)}{\sigma}=:\hat\omega_{j_0},\qquad p>1.
\end{equation*}
We now distinguish between the cases $p>2$ and $1<p<2$.
Consider first the degenerate case $p>2$. Since $\hat{\omega}_{j_0}>1$, we have
\begin{equation*}
\osc\limits_{Q_{r_{j_0}, b\,\hat\omega_{j_0}^{2-p}r^{p}_{j_0}}(x_0, \tau_0)}v \leqslant \osc\limits_{Q_{r_{j_0}, b\,r^{p}_{j_0}}(x_0, \tau_0)}v \leqslant \hat\omega_{j_0}.
\end{equation*}
Therefore, by the classical result \cite[Chapter~3, Proposition~3.1]{DiB1993},
there exist constants $\hat{\epsilon},\hat{\sigma}\in(0,1)$ such that, defining $r_{j_0+j}:=\hat{\epsilon}^jr_{j_0}$ and $\hat\omega_{j_0+j}:=\hat{\sigma}^j\hat\omega_{j_0}$, the following estimate holds for all $j\geqslant0$:
\begin{equation}\label{eq9.18}
\osc\limits_{Q_{r_{j_0+j}, b\,\hat\omega_{j_0+j}^{2-p}r^{p}_{j_0+j}}(x_0, \tau_0)}v\leqslant \hat\omega_{j_0+j}.
\end{equation}
For the singular case $1<p<2$, we observe that
\begin{equation*}
\osc\limits_{Q_{\hat\omega_{j_0}^{-\frac{2-p}{p}}r_{j_0}, b\,r^{p}_{j_0}}(x_0, \tau_0)}v \leqslant \osc\limits_{Q_{r_{j_0}, b\,r^{p}_{j_0}}(x_0, \tau_0)}v \leqslant \hat\omega_{j_0}.
\end{equation*}
Therefore, by the classical result \cite[Chapter~4, Proposition~2.1]{DiB1993}, there exist constants $\hat{\epsilon},\hat{\sigma}\in(0,1)$ such that, defining $r_{j_0+j}:=\hat{\epsilon}^jr_{j_0}$ and $\hat\omega_{j_0+j}:=\hat{\sigma}^j\hat\omega_{j_0}$, the following estimate holds for all $j\geqslant0$:
\begin{equation}\label{eq9.19}
\osc\limits_{Q_{\hat\omega_{j_0+j}^{-\frac{2-p}{p}}r_{j_0+j}, b\,r^{p}_{j_0+j}}(x_0, \tau_0)}v\leqslant \hat\omega_{j_0+j}.
\end{equation}
We now define:
$$
\hat\epsilon_0:=
\begin{cases}
1,&\text{if }p>2,\\
\hat\omega_{j_0}^{-\frac{2-p}{p}},&\text{if }1<p<2,
\end{cases}, 
\qquad \hat{r}_{j_0}=\hat\epsilon_0 r_{j_0}
\qquad \hat{Q}_{j_0}:=Q_{\hat{r}_{j_0}, b\,\hat\omega_{j_0}^{2-p}\hat{r}^{p}_{j_0}}(x_0, \tau_0).
$$
In these terms, \eqref{eq9.18} and \eqref{eq9.19} imply
\begin{equation*}
\osc\limits_{\hat{Q}_{j_0+j}}v\leqslant \hat\omega_{j_0+j}.
\end{equation*}
Returning to the original function $u(x,t)$, we obtain the following estimate
\begin{equation}\label{eq9.20}
\osc\limits_{{Q}_{j_0+j}}u\leqslant \omega_{j_0+j}.
\end{equation}
with $r_{j_0+j}:=\hat\epsilon_0\hat\epsilon_0^jr_{j_0}$ and $\omega_{j_0+j}:=\tfrac{1+\eta_0(1-\sigma)}{\sigma}\hat\sigma_0^j\omega_{j_0}$. This completes the proof of the Proposition. \end{proof}
 \noindent Now that the quantitative oscillation decay  \eqref{eq9.20} is valid for the net of shrinking cylinders $\{Q_{j}\}_{j \in \mathbb{N}}$, the existence of a locally H\"{o}lder continuous representative of $u$ follows from standard iterative arguments, see, for instance, \cite[Chapters~3 and~4]{DiB1993} for the standard argument and \cite{CCMV}, section 2.1, for the construction of the representative. This completes the proof of Theorem~\ref{th1.1}.

\vspace{.5cm}

\section*{Acknowledgements}  E. Henriques was financed by Portuguese Funds through FCT - Funda\c c\~ao para a Ci\^encia e a Tecnologia - 
within the Project UID/00013/2025, link at https://doi.org/10.54499/UID/00013/2025. S. Ciani acknowledges the support of GNAMPA (INdAM), the PNR funding of the Italian government and the department of Mathematics of the University of Bologna Alma Mater. I.~Skrypnik acknowledges funding from the Simons Foundation (SFI-PD-Ukraine-00017674).  M.~Savchenko was supported  by a grant of the National Academy of Sciences of Ukraine (project number is 0125U002854).

\newpage
{\small 
CONTACT INFORMATION

\medskip

{\bf Simone~Ciani} (simone.ciani3@unibo.it)\\
Universit\`a di Bologna Alma Mater, Piazza Porta San Donato 5, Italy\\

\medskip
{\bf Eurica Henriques} (eurica@utad.pt)\\
Centro de Matem\'atica CMAT; Polo CMAT-UTAD\\
Universidade de Tr\'as-os-Montes e Alto Douro, Vila Real, Portugal
  \\

\medskip 

{\bf Yevgeniia Yevgenieva} (yevgenieva@mpi-magdeburg.mpg.de)\\ 
Max Planck Institute for Dynamics of Complex Technical Systems, Sandtorstrasse 1, 39106 Magdeburg, Germany,\\
Institute of Applied Mathematics and Mechanics,
National Academy of Sciences of Ukraine, Gen. Batiouk Str. 19, 84116 Sloviansk, Ukraine\\

\medskip 
{\bf Mariia Savchenko} (shan\textunderscore maria@ukr.net)\\ Institute of Applied Mathematics and Mechanics,
National Academy of Sciences of Ukraine, Gen. Batiouk Str. 19, 84116 Sloviansk, Ukraine\\

\medskip
{\bf Igor I.~Skrypnik} (ihor.skrypnik@gmail.com)\\Institute of Applied Mathematics and Mechanics,
National Academy of Sciences of Ukraine, Gen. Batiouk Str. 19, 84116 Sloviansk, Ukraine\\

}

\end{document}